\documentclass[11pt]{amsart}
\usepackage[utf8]{inputenc}
\usepackage[T1]{fontenc}
\usepackage[english]{babel}
\usepackage{amsmath, mathtools, amsthm, amssymb, amsfonts, enumerate}
\usepackage{dsfont}
\usepackage{color}
\usepackage{geometry}
\usepackage{todonotes}
\usepackage{epstopdf}
\usepackage{bbm}
\usepackage{geometry}
\usepackage[format=plain, indention=1cm]{caption}

\usepackage{hyperref}
\hypersetup{
    colorlinks,
    citecolor=blue,
    filecolor=blue,
    linkcolor=blue,
    urlcolor=blue
}

\usepackage{textcomp}
\usepackage{pgfplots}
\usepackage{float}
\usepackage{tikz}
\usepackage{epsfig}
\usepackage{a4}

\usepackage{eurosym}
\usepackage{graphics}
\usepackage{caption}
\usepackage{subcaption}
\usepackage{graphicx}
\usepackage{bbm,dsfont}
\usepackage{cases}
\usepackage{mathrsfs}
\usepackage{xcolor}
\usepackage{tabularx}

 \usepackage[round]{natbib}

\newcommand{\stkout}[1]{\ifmmode\text{\sout{\ensuremath{#1}}}\else\sout{#1}\fi}

\newcommand{\hs}{\vspace{3mm}}

\newtheorem{thm}{Theorem}[section]
\newtheorem{theorem}{Theorem}[section]
\newtheorem{remark}[theorem]{Remark}
\newtheorem{assumption}[theorem]{Assumption}

\newtheorem{prop}[theorem]{Proposition}
\newtheorem{corollary}[theorem]{Corollary}
\newtheorem{definition}[theorem]{Definition}

\def\le{\leq}

\def\vphi{\varphi}
\def\rr{\mathbb{R}}
\def \eee{\mathbb{E}}

\definecolor{red}{rgb}{1.0,0.0,0.0}

\definecolor{blu}{rgb}{0.0,0.0,1.0}

\definecolor{gre}{rgb}{0.03,0.50,0.03}

\def\eps{\varepsilon}

\graphicspath{{./figs/}}

\renewcommand{\hat}{\widehat}

\title[Mean-Field Equilibrium with Irreversible Investment and Regime Switching]{A Stationary Mean-Field Equilibrium Model of Irreversible Investment in a Two-Regime Economy}
\author[A\"id]{Ren\'e A\"id}
\author[Basei]{Matteo Basei}
\author[Ferrari]{Giorgio Ferrari}

\address{R.\ A\"id: Université Paris-Dauphine, PSL Research University - Department of Economics - Place du Mar\'echal de Lattre de Tasssigny, 75 775 Paris Cedex, France.}
\email{\href{mailto:rene.aid@dauphine.psl.eu}{rene.aid@dauphine.psl.eu}}
\address{M.~Basei: EDF R\&D Paris and FiME (Finance for Energy Markets Research Centre), Paris, France.}
\email{\href{mailto:matteo.basei@edf.fr}{matteo.basei@edf.fr}}
\address{G.~Ferrari: Center for Mathematical Economics (IMW), Bielefeld University, Universit\"atsstrasse 25, 33615, Bielefeld, Germany}
\email{\href{mailto:giorgio.ferrari@uni-bielefeld.de}{giorgio.ferrari@uni-bielefeld.de}}

\date{\today}

\numberwithin{equation}{section}

\begin{document}

\begin{abstract} 
We consider a mean-field model of firms competing {\em à la Cournot} on a commodity market, where the commodity price is given in terms of a power inverse demand function of the industry-aggregate production. Investment is irreversible and production capacity depreciates at a constant rate. Production is subject to Gaussian productivity shocks, while large non-anticipated macroeconomic events driven by a two-state continuous-time Markov chain can change the volatility of the shocks, as well as the price function. Firms wish to maximize expected discounted revenues of production, net of investment and operational costs. Investment decisions are based on the long-run stationary price of the commodity. We prove existence, uniqueness and characterization of the stationary mean-field equilibrium of the model. The equilibrium investment strategy is of barrier-type and it is triggered by a couple of endogenously determined investment thresholds, one per state of the economy. We provide a quasi-closed form expression of the stationary density of the state and we show that our model can produce Pareto distribution of firms' size. This is a feature that is consistent both with observations at the aggregate level of industries and at the level of a particular industry. We establish a relation between economic instability and market concentration and we show how macroeconomic instability can harm firms' profitability more than productivity fluctuations.
\end{abstract}

\maketitle

\smallskip


\smallskip

{\textbf{OR/MS subject classification}}: Dynamic programming/optimal control: Markov; Games/group decisions: Stochastic; Inventory/production: Stochastic models.

\smallskip


\smallskip

{\textbf{MSC2020 subject classification}}: 49N80, 49L20, 91A15, 91A16.



%
%
%
%
%
%
%


\section{Introduction}
\label{sec-intro}

We investigate the long-run market structure of firms competing {\em à la Cournot} in a context where they face both idiosyncratic fluctuations and macroeconomic cycles. In particular, we are interested in the effects of macroeconomic instability on the stationary state of the market structure, described by the density of firms' size, the market concentration, and the profitability of firms. To this purpose, we consider a mean-field model of firms competing {\em à la Cournot} on a commodity market. Investment is irreversible and production capacity depreciates at a constant rate. Two types of shocks affect firms' profit and cost. First, production is subject to Gaussian shocks, representing random variations of productivity. Second, non-anticipated macroeconomic events can drive the whole market into instability (large fluctuations of productivity) or stability (small fluctuations of productivity). Such events are driven by a two-state continuous-time Markov chain. When a regime-change occurs, the values of the average level and elasticity of prices, as well as the volatility of the Gaussian shocks affecting production, can change. Firms wish to maximize expected discounted revenues of production, net of the investment and operational costs. One key aspect of our irreversible investment model is the {\em pricing rule} used by the representative firm to assess the appropriate level of investment. The commodity price is given in terms of a decreasing power function of the industry-aggregate production. However, we assume that the decision-maker uses the {\em long-term stationary price} of the commodity in order to estimate the profitability of her investment. Because there is a one-to-one correspondence between production capacity in the market and the price, the assumption made by the decision-maker is equivalent to making an estimation of the appropriate level of capacity that can be sustained in the market in the long-run. Although such a criterion does not corresponds to the future realised cash-flows that the firm can expect to make, it corresponds to a decision-making process of most of the firms involved in commodity markets. As a matter of fact, since making long-term forecasts of commodity prices is a highly risky exercise, firms reduce the complexity of the potential futures by designing {\em long-run scenarios} of prices. By considering a {\em stationary} mean-field equilibrium, we assume that firms found it easier to coordinate or to agree on the level of capacity that can be sustained in the market in the long-run, rather than on the whole path of investment decisions. The present model clearly takes the counter-step of stochastic dynamic models in high dimension, aiming at capturing most of the risk factors of an industry, as well as the differences in technologies (see \cite{Aid14} and the references therein for an overview of this modeling framework applied to electricity generation).

\hs

\noindent{\bf Mathematical results.} The considered stationary irreversible investment problem is modeled through a mean-field model with nondecreasing singular stochastic controls (see, e.g., \cite{Bertola} or Chapter 11 in \cite{Dixit94} for early contributions on irreversible investment problems with singular controls). The monotonicity requirement on the control processes well describes the irreversibility constraint of the investment policies and also allows to take into account lump-sum or singularly-continuous actions. The singular control affects linearly Markov-modulated geometric Brownian dynamics providing the evolution of the representative company's production capacity, and each unit of investment into production gives rise to proportional costs. At equilibrium, the company maximizes total expected discounted net revenues from production and sells the produced good at a regime-dependent price which is given in terms of the long-run industry-aggregate production. We are able to show existence, uniqueness and characterization of the equilibrium. This is achieved via a constructive \emph{three-step approach}. 

Firstly, for a given and fixed vector-valued parameter representing the regime-dependent long-run industry-aggregate production, we solve the singular stochastic control problem with regime-switching faced by the representative company. Solving such a dynamic optimization problem through a guess-and-verify approach (i.e., considering suitable parametric candidate solutions to the corresponding dynamic-programming equation and selecting the ``optimal'' parameters by imposing appropriate regularity conditions, i.e.\ the so-called smooth-fit and smooth-pasting conditions) is possible, but challenging in the present context (among many others, see \cite{Cadenillas}, \cite{Cadenillas2012}, \cite{Guo2005} for Markov-modulated control problems addressed via the guess-and-verify approach). As a matter of fact, the underlying Markov chain makes the dynamic programming equation result into a system of interconnected constrained ODEs, with the effect that it becomes hard to show existence and uniqueness of the solution to the highly nonlinear (and unhandy) smooth-fit and smooth-pasting equations. We therefore adopt a different approach, already employed in \cite{FeRo} (and inspired by the early contributions in \cite{KS84} and \cite{BK97}): We introduce an optimal stopping problem with regime switching; via direct probabilistic and analytic methods, we prove existence of thresholds triggering its optimal stopping rule, as well as regularity of its value function; finally, we verify that a suitable integral of the stopping problem's value function identifies with the value function of the considered singular stochastic control problem. As a by-product, we also obtain the form of the optimal investment rule. This prescribes to exert the minimal amount of effort needed to prevent that the (optimally controlled) production capacity falls below an endogenous trigger, depending on the current regime and, clearly, on the fixed stationary industry-aggregate production. 

As a second step, still for a given and fixed vector-valued parameter representing the regime-dependent long-run industry-aggregate production, we determine the joint stationary distribution of the optimally controlled production process and of the underlying Markov chain. This is obtained by solving the corresponding stationary Fokker-Planck equation which, in the present setting, corresponds to a system of interconnected ODEs subject to suitable boundary conditions (see also \cite{DAuriaKella}). It is worth emphasizing here, that we do obtain a semi-explicit formula for such a stationary distribution, and not only its existence and uniqueness.

Finally, we impose the consistency condition, that is, we impose that the vector-valued parameter $Q$ fixed in the previous two steps indeed identifies with the regime-dependent long-run industry-aggregate production. This naturally leads to a fixed point problem, in that the stationary distribution, and therefore its average, depend on $Q$. We address the question of existence of a solution to the fixed point problem via the Brower's fixed point theorem, while uniqueness follows from a contradiction argument, inspired by \cite{Weintraub22}, which exploits a suitable monotonicity property of the investment triggers with respect to $Q$. It is worth noticing that such a monotonicity can be easily shown via the relation to optimal stopping and it is not implied by the well-known Lasry-Lions monotonicity condition (see, e.g., p. 169 of volume I in \cite{CarmonaDelarue}), as our requirements on the instantaneous profit and inverse demand functions make the Lasry-Lions condition not fulfilled in general (see also Remark 1 in \cite{Cao22}).

Our constructive approach to the existence and uniqueness has the important by-product that it also yields a complete characterization of the equilibrium itself. In particular, the equilibrium regime-dependent investment-triggers and prices are completely determined through a system of nonlinear algebraic equations. Those can be then solved numerically in order to understand the economic insights of our model.


\hs

\noindent{\bf Economic results.} In the course of the construction of the mean-field equilibrium, a first important result concerns the semi-explicit determination of the stationary distribution of the firms' sizes, which follows a Pareto law, $\mathbb P\big( X_{\infty} \geq x \big) \sim x^{-|\theta_2|}$ for some tail parameter $|\theta_2|$(here, and in the following, with a slight abuse of notation, $X_{\infty}$ is a random variable distributed according to the stationary distribution of the equilibrium state-process). It is a well-documented stylized fact in Industrial Economics (see \cite{Axtell01}) that at the aggregate level (i.e.\ mixing all types of firms in the same sample), the tail parameter is close to one. In that case, the distribution becomes a Zipf's law (see \cite{Gabaix99}). According to the model of endogenous firms' growth based on innovation developed by \cite{Luttmer07}, a potential explanation for this value is the small imitation cost across sectors and firms. In our model, the tail parameter $\theta_2$ depends only on the volatilities of the states, the depreciation rate of the capital, and on the Markov chain's intensities of jumps. When intensities of switches are small ($p_1=p_2 \approx 0$), we recover the power law coefficient as in \cite{Luttmer07}(see p.~1125 therein). Besides, we find that in sectors where capital slowly erodes, the power law turns closer to a Zipf's law. However, in contrast, our model can also induce large deviations from one, if, for instance, the capital erodes quickly (see Section~\ref{ssec:distrib}). As a matter of fact, at the sectoral level, it is possible to exhibit power law coefficients larger than one.  In \cite{Rossi07} it is showed that firms' size tail distribution depends on the capital intensity usage, both physical and human capital.  Recent empirical results at different industries level exhibit large negative exponent, like in \cite{Halvarsson14}(Figure~1), where a coefficient around $-4$ can be found in some industries (see also \cite{Bee17}). 

Second, we investigate the market concentration in the stationary state. Indeed, in a general setting of stationary mean-field models, \cite{Adlakha15} find that the light-tail feature of the stationary distribution of players' size is a sufficient condition for the existence of a stationary mean-field equilibrium. In our model, we have seen that equilibrium firms' size exhibit fat-tails, thus showing that the sufficient condition of \cite{Adlakha15} is not necessary. Hence, in order to measure the extent of market concentration in a setting with an continuum of firms, we introduce two indices. The first one is a version of the Herfindahl-Hirschman index (HHI), when the number of firms goes to infinity. The HHI index is used in market concentration analysis by most regulators and it is defined as $H_n := \sum_{i=1}^n s_i^2$, where $n$ is the number of firms serving the market and $s_i$ is the market share of firm $i$. Fully fragmented and highly concentrated markets both exhibit an average market share going to zero as $n$ goes to infinity. However, in the first case, the variance of market shares is constantly equal to zero, while the ratio between variance and expectation in the latter case admits a finite limit (see Section~\ref{ssec:market}). Those remarks suggest using as an index of market concentration the ratio between the variance of the firms' size and the square of its expectation at the stationary equilibrium. The second index we choose is a Gini index, as already suggested in \cite{Hopenhayn92}. For a given quantile $q \in (0,1)$, define $x(q)$ as the lowest $x$ such that $F(x) := \mathbb P(X_\infty\leq x) = q$. Then we define the Gini curve by $\bar Q(q) = \mathbb E[X_\infty | X_\infty \leq x(q)]/{\overline{Q}}^\star$~$\in (0,1)$, where ${\overline{Q}}^\star$ is the equilibrium average production across the regimes. Finally, the Gini index of market concentration is defined as 
\begin{align*}
H := \int_{0}^1 \big| q -  \bar Q(q) \big|  dq.
\end{align*}
The H-index measures deviations from the uniform distribution of market shares, as measured by the capacity held by firms at each level of quantiles. A fully fragmented market would yield a zero H-index, whereas a fully concentrated market served by a single monopolistic firm would induce an H-index of 1/2.  We inquiry the effect of economic instability as measured by the increase of $p_1$ on market concentration. As a matter of fact, $1/p_1$ is the average time spent by the economy in the regime with larger volatility of production. We find that, both indicators, HHI and Gini H-index, show consistent results: a longer period of unstable economy tends to increase market concentration.

Third, we analyze the value of stability (see Section~\ref{ssec:crisis}). This point relates to the result of \cite{Lucas77} on the irrelevance of social hedging against the cost of volatility growth. Using a simple model of intertemporal maximization of utility of a risk-averse representative agent, Lucas finds that excessively high risk-aversion would be necessary to justify the economic interest to hedge society against macroeconomic fluctuations. This result was reassessed in the macroeconomic literature, in particular in \cite{Epaulard03}, where the authors reach the same conclusion using more recent growth data and an endogenous growth model. Our model is only a partial equilibrium model. Nevertheless, it allows to investigate at a sectoral level the relative effects of Gaussian fluctuations of productivity compared to global macroeconomic shocks. To do so, using our quasi-analytical solution of the stationary state density, we compute the value $V ^\star$ of the firms' expected profit at the stationary equilibrium. Consider the vector $\nu_i := (\sigma_i,p_i,\phi_i)$ of the volatility of production shocks $\sigma_i$, the intensity of switch $p_i$, and the level of price $\phi_i$. We evaluate the elasticities of $V^\star$ w.r.t.~the State 1 volatility~$\sigma_1$, intensity of switch~$p_1$, and the average level of price~$\phi_1$ at a point where $\nu_1 = \nu_2$. That is, we assess the percentage effect of a slight increase in quality of State 1, compared to State 2 on the profitability of the representative firm. Using standard value parameters for depreciation and discount rate, we find that increasing the volatility $\sigma_1$ by 1\% decreases the stationary value of the firm by 0.08\%, but reducing the price level $\phi_1$ by 1\% cuts the value $V^\star$ by nearly the same amount. This result provides a quantification of the intuition that a small reduction on the average selling price is a larger disaster at sectoral level than a small increase of productivity shocks volatility. Besides, we observe a quasi-constant and close to $1$ elasticity of $V^\star$ w.r.t.\ the level of price, as well as a lower but sharply increasing elasticity w.r.t.\ the volatility. It means that the marginal cost of volatility is increasing while the marginal gain from stability is decreasing.

\hs

\noindent{\bf Related literature.} Our paper belongs to the literature on the dynamics of investment of firms, in particular when investment is irreversible and represented as a singular control process. This vast literature includes the early \cite{BK97} and \cite{Bertola}, and the more recent \cite{Aid15}, \cite{DeAFeFe}, \cite{Ferrari}, and \cite{RiSu}, and the references therein. A close related paper is \cite{Grenadier02}, where a singular control equilibrium problem with homogeneous firms competing {\em à la Cournot} similar to ours can be found. In his setting, Grenadier provides the $N$-firm equilibrium and gives explicit solution for the investment threshold for different classical dynamics of the demand shocks. 
In \cite{BackPaulsen} and \cite{Steg}, $N$-player capital accumulation games with open loop strategies are considered in general Markovian and non-Markovian settings, respectively. On the other hand, \cite{Kwon} proposes a two-player singular stochastic control game to model a public good irreversible contribution game and analyze the gradualism arising from the free rider effect (see also \cite{Ferrarietal}). 

With reference to the literature on mean-field games and competitive market equilibria with a continuum of agents, our paper is placed amongst those works that study mean-field equilibria for games with singular controls. Amongst those, \cite{Miao08} presents an analytically tractable competitive equilibrium model and study the effect of frictions, such as irreversibility and fixed costs, on the long-run equilibrium; \cite{BertolaCaballero} propose and solve a model of sequential irreversible investment and study theoretically and empirically the aggregate implications of microeconomic irreversibility and idiosyncratic uncertainty; \cite{CaoGuo} solve a stationary discounted mean-field game with two-sided singular controls, and analyze its relation to the associated $N$-player game; \cite{Cao22} address a general class of stationary one-dimensional mean-field games with discounted and ergodic criteria and study the relation between the resulting equilibria; in \cite{Campietal} and \cite{GuoXu} mean-field and $N$-player stochastic games for finite-fuel follower problems are studied, and the structure of equilibria is obtained. Finally, \cite{HorstFu} provide a careful technical analysis of the question of existence for general mean field games involving singular controls. 

Our model focuses on the determination and characterization of a stationary equilibrium. In this regard, our work relates more generally to those treating stationary mean-field games and stationary oblivious equilibria for infinite models (cf.\ \cite{Adlakha15, Bardi, Hopenhayn92, Weintraub08, Weintraub11}, among others), where it is assumed that the representative player makes actions only on the basis of her own state and the long-run average state of the mass. This formalizes the following idea: In a symmetric game with a large number of players, whose state and performance criterion only depend on the distribution of opponents' state (i.e.\ an anonymous game, cf.\ \cite{JovanovicRos}), fluctuations of players' states are expected to average out, the behavior of the other agents is ``lost in the crowd'', and the population's state remains roughly constant over time.  

\hs

\noindent{\bf Structure of the paper.} The paper is organized as follows. Section~\ref{sec-formulation} provides a detailed description of the model, its assumptions and gives the main result of existence and uniqueness of the equilibrium together with the closed-form expression of the equilibrium stationary distribution of the state-process.  Section~\ref{sec-study} builds on the former section to provide first illustrations of the solution (Section~\ref{ssec:dynamics}), and then results on the density of firms (Section~\ref{ssec:distrib}), on the market concentration (section~\ref{ssec:market}) and on the profitability of firms (section~\ref{ssec:crisis}). Section~\ref{sec:Equiconstr} provides the proof of the main result. Finally, Appendix \ref{sec-appendix} collects some technical proofs.

\section{The Model and the Main Result}
\label{sec-formulation}

There is a continuum of firms of unitary mass indexed by their production capacity. Firms behave competitively, taking prices of output and input as given. Firms are ex ante identical in that their technology or productivity shocks are drawn from the same distribution. They differ ex post in the realization of idiosyncratic shocks. This is modeled as a one-dimensional Brownian motion $B=(B_t)_{t \geq 0}$ on a given complete probability space $(\Omega, \mathcal{F}, \mathbb{P})$. We denote by $\eee$ the expectation under $\mathbb{P}$. On the same probability space, it is also defined a two-state irreducible continuous-time Markov chain $\eps=(\eps_t)_{t\geq0}$, with irreducible generator $P=(P_{ij})_{i,j=1,2}$ and stationary distribution $\pi$:
\begin{equation}
\label{eq-Pi}
P:= \begin{bmatrix}
-p_1 & p_1 \\
p_2& -p_2 
\end{bmatrix}
, \qquad\quad
\pi=(\pi_1,\pi_2):=\Big(\frac{p_{2}}{p_1+p_2},\, \frac{p_{1}}{p_1+p_2}\Big),
\end{equation}
with $p_1,p_2 \in (0,1)$. In particular, the time spent in state $i$ before switching to state $j\neq i$ is an exponential random variable with parameter $p_i$.

We assume that $B$ and $\eps$ are independent and we denote by $\mathbb{F} := (\mathcal{F}_t)_{t\geq0}$ the filtration generated by $(B_t, \eps_t)_{t \geq 0}$, as usual augmented by the $\mathbb{P}$-null sets of $\mathcal{F}_0$. While the Brownian motion $B$ drives the stochastic component of the production and is responsible, e.g., of productivity shocks, the two-state Markov chain $\eps$ models the regime of the economy. 

A representative company's production capacity evolves as 
\begin{equation}
\label{eq-SDE}
dX^I_t = - \delta X^I_t dt + \sigma_{\eps_t} X^I_t dB_t +  X^I_t \circ dI_t, \qquad t \geq 0, \qquad X_{0-}=x>0,
\end{equation}
where $\delta, \sigma_1,\sigma_2>0$ are given positive constants and $I = (I_t)_{t \geq 0} \in \mathcal{M}$, where 
\begin{equation*}
\mathcal{M} := \{\text{$I:\Omega\times[0,\infty) \to [0,\infty)$, $\mathbb{F}$-adapted, with $t \mapsto I_t$ non-decreasing, c\`adl\`ag and $I_{0-}=0$ a.s.}\}.
\end{equation*}
The first term on the right-hand side of \eqref{eq-SDE} corresponds to depreciation, due to the ageing of the means of production; the second term models production uncertainty, with the amplitude of the Brownian shocks depending on the current regime of the economy; the third term is due to the company's irreversible investment into production. As a matter of fact, $I_t$ represents the cumulative investment (per unit of production) performed over the time period $[0,t]$, so that $dI_t$ represents, informally, the amount of investment, per unit of production capacity, made in the infinitesimal amount of time $dt$. 

More precisely, given that any $I \in \mathcal{M}$ can be decomposed as $I_t = I^c_t + I^j_t$, where $I^j_t := \sum_{s\leq t, \Delta I_s \neq 0} \Delta I_s$, with $\Delta I_s := I_s - I_{s-}$, is the discontinuous part of $I$ and $I^c_t:=I_t-I^j_t$ is its continuous part, we follow \cite{Zhu}, \cite{AlMotairiZervos} and \cite{GuoZervos}, among others, and define 
\begin{equation}
\label{eq-XcircI}
\int_{[0,\cdot]} X^I_t \circ dI_t := \int_0^{\cdot} X^I_t dI^c_t + \sum_{t \leq \cdot} X^I_{t-} \int_0^{\Delta I_t} e^{ u} du = \int_0^{\cdot} X^I_t dI^c_t + \sum_{t \leq \cdot} X^I_{t-} \big(e^{\Delta I_t} - 1\big).
\end{equation}
In order to justify \eqref{eq-XcircI}, assume that a small intervention $h$ implies a proportional jump, i.e.\ $X_t = (1+h)X_{t-} \sim e^{ h}X_{t-}$. Thinking of any intervention $\Delta I_t$ as the combination of $N$ small interventions of size $h=\Delta I_t/N$, this in turn leads to $X_t = (e^{ h})^N X_{t-} = e^{\Delta I_t} X_{t-}$, hence $\Delta X_t = (e^{\Delta I_t} -1)X_{t-}$. Thanks to \eqref{eq-XcircI}, an application of It\^o's formula implies that \eqref{eq-SDE} admits the following solution:
\begin{equation}
\label{eq-SDE-sol}
X^I_t = x \, \exp\bigg(-\Big(\delta t + \frac12 \int_0^t \sigma^2_{\eps_s} ds\Big) + \int_0^t \sigma_{\eps_s} dB_s +  I_t \bigg).
\end{equation}

Let $\eta_i$ be the unitary market price of the company's production (goods, commodities, ...) when the economy is in regime $i \in \{1,2\}$, so that, assuming that production happens at full capacity and that demand and offer are in equilibrium, the company's revenue from the sale of its production at time $t$ is then $\eta_{\eps_t} X^I_t$. Furthermore, we assume that production gives rise to running costs that are quadratic in the production capacity, while investment cost is proportional to the invested amount, with marginal cost $\kappa>0$. 

Hence, given $(x,i) \in \rr_+ \times \{1,2\}$ be a given initial state, for fixed $Q \in \rr^2_+$ and $I \in \mathcal{A}$, with
\begin{equation}
\label{set:A}
\mathcal{A} := \Big\{ I \in \mathcal{M}:\, \lim_{T\uparrow \infty}\eee_{(x,i)}\big[e^{-\rho T}|X^I_T|^2\big]=0 \quad \text{and} \quad \eee_{(x,i)} \bigg[\int_0^\infty e^{-\rho t} \Big((X^I_t)^2 dt + X^I_t \circ dI_t \Big) \bigg] < \infty  \Big\},
\end{equation}
the company faces the following net profit functional:
\begin{equation}
\label{eq-J}
J_{(x,i)}(I,Q) := \eee_{(x,i)}\bigg[\int_0^\infty e^{-\rho t} \Big(\eta_{\eps_t} X^I_t  dt - c (X^I_t)^2 dt -  \kappa X^I_t \circ dI_t \Big) \bigg].
\end{equation}
Above and in the sequel, $\rho,c>0$ are given constants and $\eee_{(x,i)}$ denotes the expectation conditioned on $(X^I_{0-}, \eps_{0-})=(x,i)$.

The next technical condition requires that the representative agent is sufficiently impatient, and it will be relevant when proving the admissibility of a candidate equilibrium irreversible investment. 
\begin{assumption}
\label{ass:assrho}
$$\rho > 2\max\{\sigma^2_1,\sigma^2_2\}.$$
\end{assumption}
Notice that, under Assumption \ref{ass:assrho}, one has $\rho + 2\delta > \max\{\sigma^2_1,\sigma^2_2\}$, which suffices in order to ensure that the control $I\equiv 0$ belongs to $\mathcal{A}$, and therefore $\mathcal{A}\neq \emptyset$.
 
The interaction of the representative company with the industry comes through the price at which the produced good is sold. In particular, we assume that, for any regime $i \in \{1,2\}$, prices are given through the inverse-demand relation 
\begin{align}
\label{eq-eta}
\eta_i & = \vphi_i + \zeta_i Q_i^{-\alpha}, 
\end{align}
where $\vphi_1, \vphi_2, \zeta_1,\zeta_2 > 0$ and $0<\alpha<1$ are given constants such that $\min\{\vphi_1, \vphi_2\} > \rho + \delta$ \footnote{This means that the price of the good is strictly larger than the price arising in a perfectly competitive market, i.e.\ $\rho + \delta$. In particular, such a condition on the exogenous parameters is sufficient to ensure a trade-off between costs of production and investment and revenues.} and $Q_i$ gives the stationary aggregate production of the industry in regime $i \in \{1,2\}$. This is clarified through the following Definition of equilibrium.
\begin{definition} 
\label{def-MFE}
The pair $(I^{\star},Q^{\star}) \in \mathcal{A} \times \rr^2_+$ is a \textbf{stationary mean-field equilibrium} (MFE) for the model with data $(x,i) \in \rr_+ \times \{1,2\}$ if:
\begin{itemize}
    \item[(i)] $I^{\star}$ maximizes $J_{(x,i)}(\,\cdot\,, Q^{\star})$; that is,
    $$J_{(x,i)}(I^{\star},Q^{\star}) \geq J_{(x,i)}(I,Q^{\star}), \qquad I \in \mathcal{A};$$
    \item[(ii)] the pair $(X^{I^{\star}}_t,\eps_t)_{t \geq 0}$, formed by the optimally controlled production capacity and the Markov chain, admits a stationary distribution $(p_\infty(dx,i))_{i=1,2}$ and, letting $(X_\infty^{I^{\star}},\eps_\infty) \sim p_\infty$, we have
    \begin{equation*}
    Q_i^{\star} = \frac{1}{\pi_i} \int_0^\infty x \, p_\infty(dx,i), \qquad i=1,2. 
    \end{equation*}
\end{itemize}
\end{definition}

\begin{remark}
\label{rem:equilibrium}
Recalling that the mass of the continuum of companies has been normalized to one, $\frac{1}{\pi_i}p_\infty(dx,i)$ in Definition \ref{def-MFE} represents the equilibrium number (i.e.\ the equilibrium percentage) of companies that in the long-run have production capacity between $x$ and $x+dx$, when the regime of the economy is $i\in\{1,2\}$. In particular, this allows to equivalently write
\begin{equation}
\label{eq:Qiprobab}
Q_i^{\star}=\eee\big[ X_\infty^{I^{\star}} \big| \eps_\infty = i \big],\quad i \in \{1,2\},
\end{equation}
where the random variable $(X_\infty^{I^{\star}},\eps_\infty) \sim p_\infty$ and $\eps_\infty \sim \pi$ (cf.\ \eqref{eq-Pi}).
\end{remark}

Under technical integrability requirements, we can prove that a unique mean-field equilibrium indeed exists. The following statement is presented in an informal way and it is just meant to provide the necessary information on the equilibrium structure needed for the numerical analysis developed in the next Section \ref{sec-study}. The constructive proof of the existence and uniqueness claim will be then distilled in Section \ref{sec:Equiconstr}.

\begin{theorem}
\label{thm:Equilibrium}{\rm [Equilibrium existence, uniqueness and structure]}
Let Assumption \ref{ass:assrho} hold and, for $i\in\{1,2\}$, define $\phi_i(\theta) := \frac12 \sigma_i^2\theta^2 + \Big(\delta + \frac12 \sigma^2_i\Big)\theta - p_i$ and denote by $\theta_2$ the largest negative root of the equation $\phi_1(\theta)\phi_2(\theta) - p_1p_2 = 0$. If $\theta_2 < -1$, then there exists a unique stationary mean-field equilibrium in the sense of Definition \ref{def-MFE}. In particular:
\begin{itemize}
\item[(i)] The equilibrium investment strategy $I^{\star}$ is given by
\begin{equation*}
I^{\star}_t = 0 \,\, \vee \,\,   \sup_{0 \leq s \leq t} \bigg( \ln(a_{\eps_s}/x) + \int_0^s (\delta +\frac12 \sigma_{\eps_u}^2) du - \int_0^s \sigma_{\eps_u} dB_u \bigg), \qquad t \geq 0, \qquad I^{\star}_{0-}=0.
\end{equation*}
It prescribes to keep the equilibrium production capacity above a regime-dependent barrier at all the times via an upwards reflection: $X^{I^{\star}}_t \geq a^{\star}_{\eps_t}$. In particular, the investment should be the minimal one that prevents the production leaves the region $\{(x,i)\in \rr_+ \times \{1,2\}:\, x \geq a^{\star}_i\}$. The barriers $a^{\star}_i$, $i\in \{1,2\}$, are endogenously determined and uniquely solve a system of nonlinear algebraic equations (cf.\ \eqref{eq-coeff-a} below).
\vspace{0.2cm}

\item[(ii)] The stationary distribution of $(X^{I^{\star}}_t,\eps_t)_{t \geq 0}$ admits a density with respect to the Lebesgue measure, it is explicitly computable (see Corollary \ref{cor:exlicitPihat} below) and $\mathbb{P}(X^{I^{\star}}_{\infty} \geq x) \sim x^{\theta_2}$.
\end{itemize}
\end{theorem}

The optimal investment policy is characterized by \emph{small-scale actions} and \emph{large-scale actions}. The former are employed as soon as, in absence of a regime switch, the production capacity $X^{I^{\star}}_t$ attempts to fall below the boundary $a^{\star}_{\eps_t}$. The purpose of these continuous actions is to make sure, with a minimal effort, that $X^{I^{\star}}_t$ is kept inside the interval $[a^{\star}_{\eps_t}, \infty)$. On the other hand, large-scale actions are lump-sum investments whose purpose is to bring $X^{I^{\star}}_t$ back to level $a^{\star}_{\eps_t}$ through a jump. Note that, those lumpy interventions are only needed at times of jumps of the macroeconomic regime switching process $\eps$, and possibly at initial time. Both small-scale and large-scale actions can be observed in Figure \ref{fig:dynamics}, introduced and described in the next section. 

\begin{remark}
\label{rem:modelextension}
It is worth noticing the existence and uniqueness result could actually be derived in a more general setting. For example, the instantaneous profit function of the representative company could have been taken to be a general concave function of the production capacity (not necessarily quadratic), increasing in the price variable and satisfying suitable growth conditions, and the inverse demand function to be a positive, nonincreasing function of the aggregate production (not necessarily of power type). However, since the resulting equations would become more complex and, on the other hand, no additional insights would be added, we stick on the linear-quadratic framework presented in this section.
\end{remark}


\section{Numerical Analysis and Economic Implications}
\label{sec-study}

\subsection{Investment dynamics}
\label{ssec:dynamics}

We illustrate the behaviour of the investment dynamics induced by our model\footnote{Throughout Section \ref{sec-study} we denote, with a slight abuse of notation,  $X_{\infty} := X_\infty^{I^{\star}}$.}. Figure~\ref{fig:dynamics}~(a) shows a single trajectory of the process $X$. As expected, the capacity is maintained over the threshold $a^\star_1$ (blue) until the state of the economy switches to State 2. At that point, the decision-maker stops compensating the depreciation of the assets and let it falls until it reaches the lower level of investment threshold $a^\star_2$ (red). We observe in Figure~\ref{fig:dynamics}~(b) that the long-run averages of the process $X$ reflected on the boundaries $a^\star_i$ tends to the stationary production capacities level $Q^\star_i$. Figure~\ref{fig:a1a2}~(a) presents the investment thresholds $a^\star_i$ as a function of the volatility $\sigma_1$ of State 1. We observe that as long as $\sigma_1$ is smaller than $\sigma_2$, $a^\star_1$ is larger than $a^\star_2$.  This means that the higher the volatility, the longer the decision-maker waits to invest, which is consistent with the standard real option theory results on the value to wait (see \cite{McDonald86}). Figure~\ref{fig:a1a2}~(b) gives the (stationary) probability for a firm to be stuck between the two investment thresholds $a^\star_1$ and $a^\star_2$, as well as the percentage of capacity $\chi_{\infty}$ that at equilibrium is stuck between those two thresholds. As expected, the probability is decreasing for $\sigma_1$ lower than $\sigma_2$ and increasing afterwards because the interval $(a^\star_1,  a^\star_2)$ is first reducing and then expanding. We observe that although the probability is increasing for $\sigma_1 > \sigma_2$, the share of capacity stuck in the corridor $(a^\star_1,  a^\star_2)$ reaches a maximum and then decreases. It means that there is an increasing proportion of firms in that interval with always smaller sizes. This observation is linked to our subsequent analysis of the distribution of the firms' size and of the percentage of capacity in the tail of the distribution.

\begin{figure}
\begin{center}
\begin{tabular}{c c c c}
(a) & (b)  &\\
\includegraphics[width=0.45\textwidth]{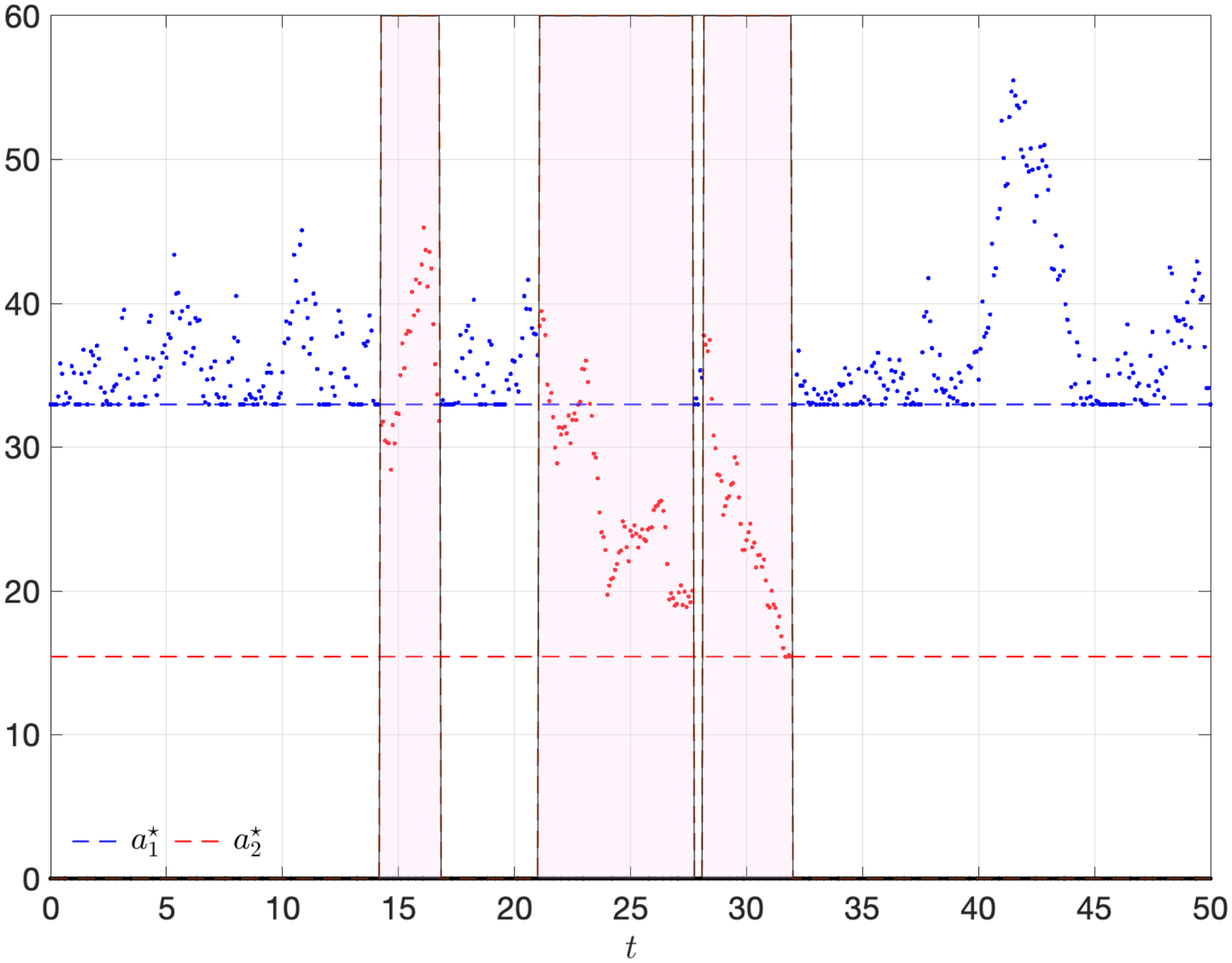} &
\includegraphics[width=0.45\textwidth]{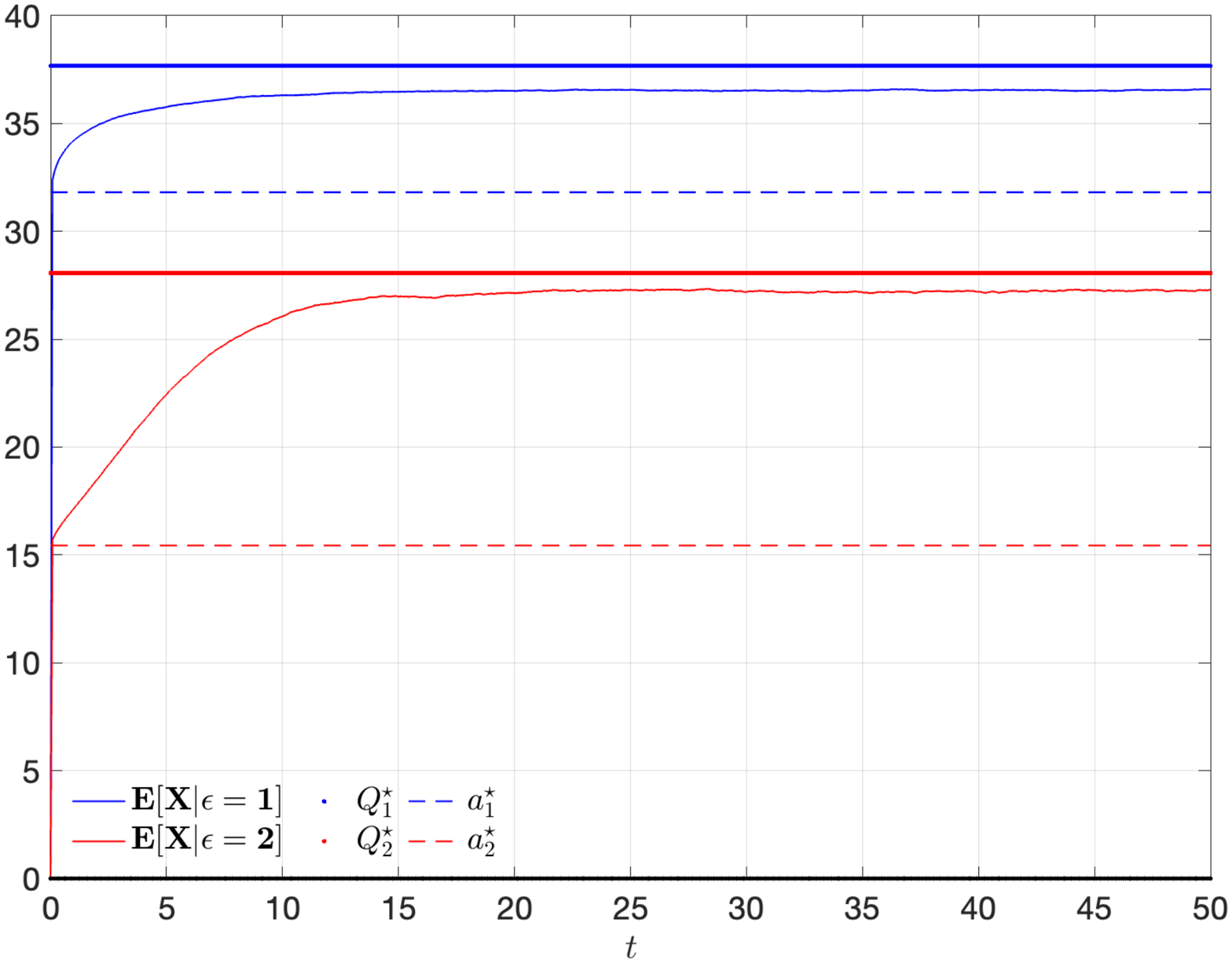}\\
\end{tabular}
\caption{ {\small (a) A trajectory of the state process $X$;  (b)  $\mathbb{E}[ X_\infty | \eps_\infty=1]$ and $\mathbb{E}[ X_\infty | \eps_\infty=2]$. Parameters: $\delta=0.1$, $\rho=0.08$,  $\kappa=10$, $c=0.1$, $\phi_1=10$, $\phi_2=5$, $\zeta_1=\zeta_2=1$, $\sigma_1=0.2$, $\sigma_2=0.15$, $\alpha=0.5$, $p_1=1/10$, $p_2=1/5$.}}
\label{fig:dynamics}
\end{center}
\end{figure}

\begin{figure}
\begin{center}
\begin{tabular}{c c }
(a) & (b)  \\
\includegraphics[width=0.45\textwidth]{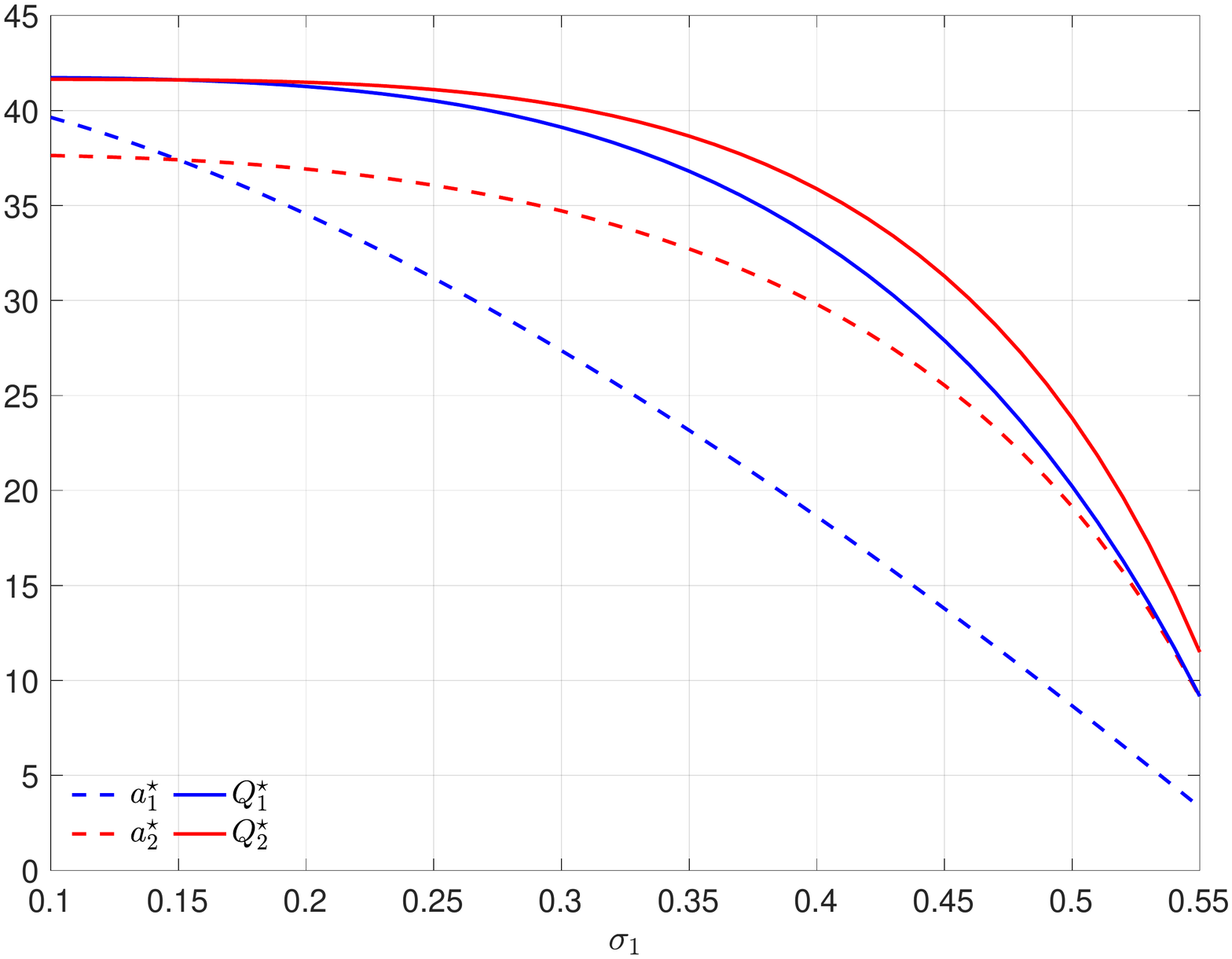} &
\includegraphics[width=0.45\textwidth]{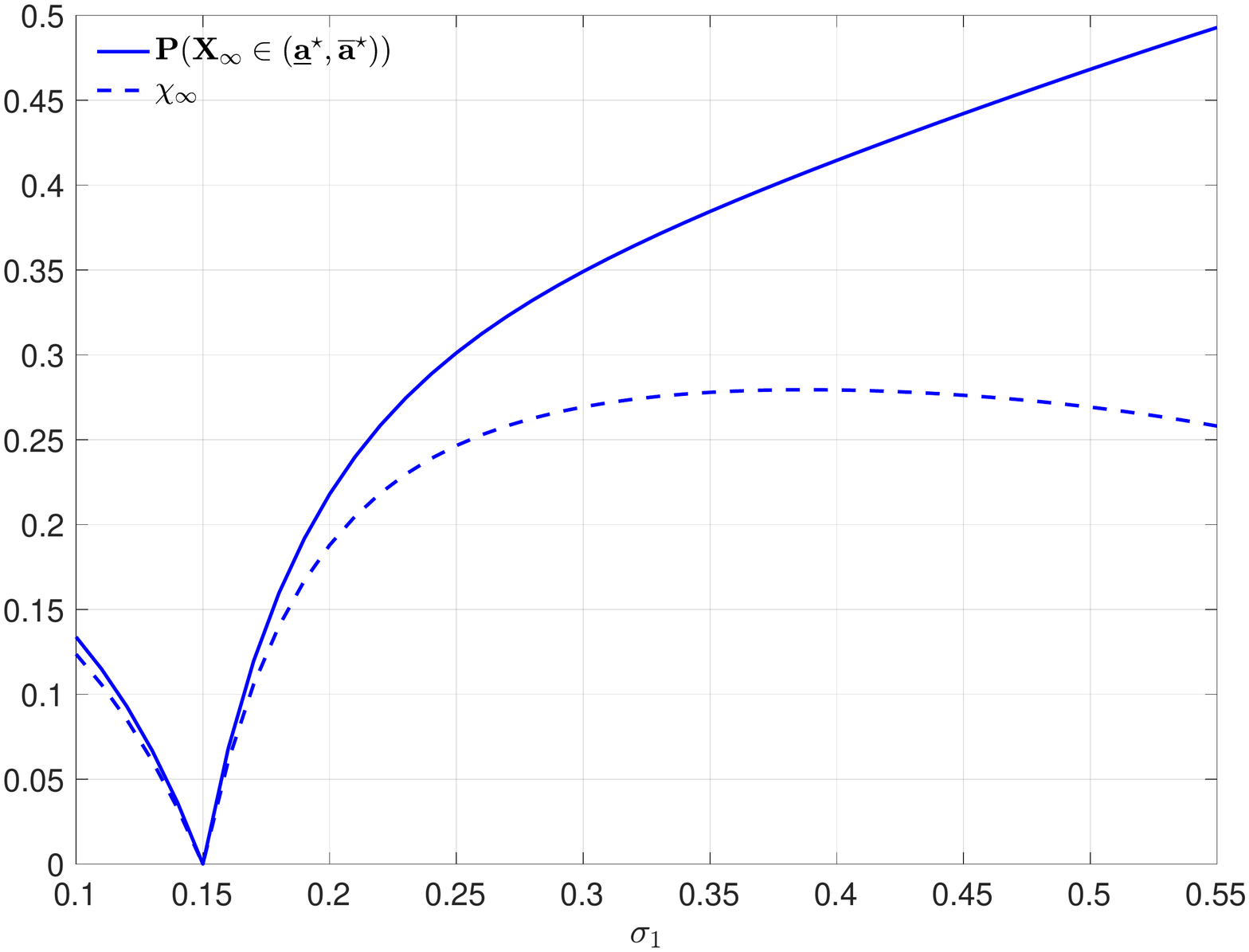}
\end{tabular}
\caption{  {\small As a function of the volatility $\sigma_1$ in State 1: (a) Investment thresholds $a_1^\star$, $a_2^\star$ and long-run equilibrium capacities $Q_1^\star$, $Q_2^\star$; (b) Ratio $\chi_\infty$ and $\mathbb P\big( X_\infty \in (\underline{a}^\star, \overline{a}^\star) \big)$, where $\underline{a}^\star := \min\{ a^\star_1, a^\star_2 \}$ and $\overline{a}^\star := \max\{ a^\star_1, a^\star_2 \}$. Parameters: $\delta=0.1$, $\rho=0.08$,  $\kappa=10$, $c=0.1$, $\phi_1=10$, $\phi_2=10$, $\zeta_1=\zeta_2=1$, $\sigma_2=0.15$, $\alpha=0.5$, $p_1=1/10$, $p_2=1/5$.}}
\label{fig:a1a2}
\end{center}
\end{figure}

\begin{figure}
\begin{center}
\begin{tabular}{c c c c}
(a) & (b)  &\\
\includegraphics[width=0.45\textwidth]{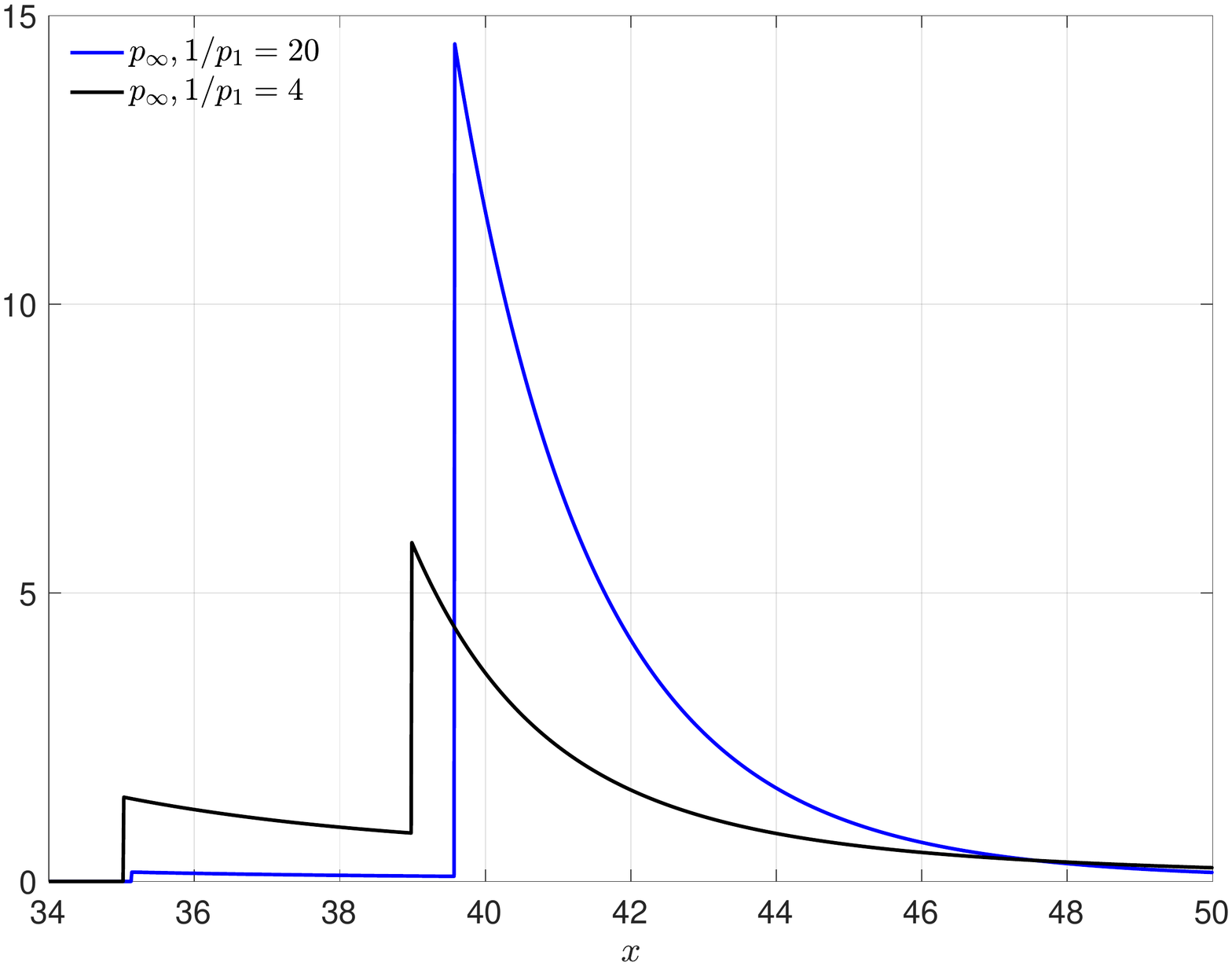} &
\includegraphics[width=0.45\textwidth]{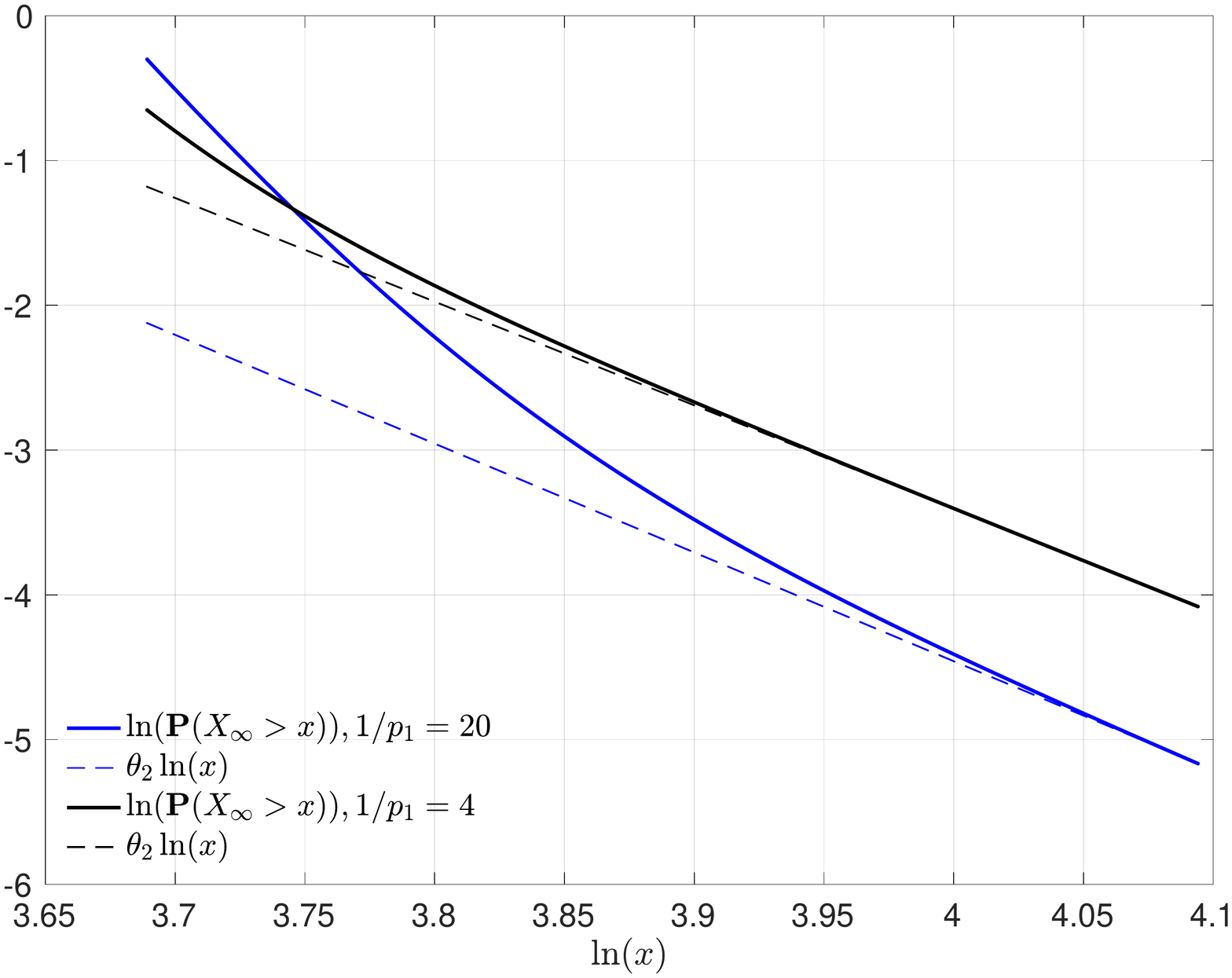}
\end{tabular}
\caption{ {\small (a)  Marginal stationary density $p_\infty(x)$; (b) Tail density. Parameters: $\delta=0.1$, $\rho=0.08$,  $\kappa=10$, $c=0.1$, $\phi_1=10$, $\phi_2=5$, $\zeta_1=\zeta_2=1$, $\sigma_1=0.2$, $\sigma_2=0.15$, $\alpha=0.5$, $p_2=1/5$. The value $p_1=1/20$ leads to $\theta_2=-7.51$ and the value $p_1=1/4$ leads to $\theta_2=-7.16$.}}
\label{fig:densities}
\end{center}
\end{figure}

\subsection{Distribution of firms' size}\label{ssec:distrib}

As the equilibrium dynamics of the representative firm are of reflected geometric Brownian motion type, it is not surprising to observe power laws for their stationary distribution. This observation can be found in the survey of \cite{Gabaix99} on Pareto's law in Economics with application in spatial economy of cities and in \cite{deWit05}'s paper providing an overview of possible probability distributions of firms' size. In our model, because of a two-regime common macroeconomic shock affecting all firms, the induced stationary densities $p_\infty(\cdot,i)$ shown in Figure~\ref{fig:densities}~(a) exhibit a bimodal shape. The lower $1/p_1$, the lower the time that is spent by the system in the unstable State 1 and thus, the more pronounced is the bimodal shape of the stationary density. However, the coefficient of the tail distribution  is fully given by the parameter $|\theta_2|$ in \eqref{eq:thetas} (see~Figure~\ref{fig:densities}~(b)). Since this is a solution to the polynomial $\bar \phi(x) := \phi_1(x)\phi_2(x) - p_1p_2$ with $\phi_i(x) = \frac12 \sigma_i x^2 + \big(\delta+\frac12 \sigma_i^2\big)x-p_i$, the parameter $\theta_2$ only depends on the volatilities of the states, the depreciation rate of the capital and on the Markov chain intensities of jumps. Convergent empirical studies attest the presence of a power law in the distribution of the size of firms. At the aggregate level, \cite{Axtell01}, using a large sample of firms' size, finds a coefficient close to one. The value of this coefficient can be explained by small imitation cost across sectors and firms, according to  the model of endogenous growth firm based on innovation developed by \cite{Luttmer07}. At the sectoral level it is possible to exhibit differences in power law coefficients. \cite{Rossi07} show that firms' size's tail distribution depends on the capital intensity usage, both physical and human capital. 

In our model, when intensities of switches are small ($p_1=p_2 \approx 0$), we recover the power law coefficient as in \cite{Luttmer07} (see p.~1125 therein), since then $\theta_2 \approx - 1 - 2  \delta/\min\{\sigma_1; \sigma_2\}$. And thus, we find that in sectors where capital slowly erodes (i.e.\ $\delta$ is small enough), the power law turns closer to a Zipf's law. Nevertheless,  even in a capital intensive industrial sector like electricity generation, we have roughly $\delta = 1/40$ per year and a volatility of production of $\sigma_{1,2} \approx 0.15$ per year (unplanned outage rate per year), which leads to $|\theta_2| \approx 1.3$. It is already a significant deviation from one. In contrast, our model can induce large deviation from $\theta_2 = -1$. Assume the same average price in both the regimes of the economy $\zeta_1 = \zeta_2$ and large potential differences in volatilities: $\sigma_1=0.1$ while $\sigma_2 =0.3$. State 1 appears as a stable state whereas State 2 is more unstable. In that case, the increasing persistence of State 1 by making $p_1$ smaller leads to large power law coefficients $|\theta_2|$, close to 7 (see Figure~\ref{fig:densities}~(b)). Empirical results at the industry level show possible large negative exponent, like in \cite{Halvarsson14}(Figure~1) where a coefficient around $-4$ can be found in some industries (see also \cite{Bee17}).

\subsection{Market concentration and fragmentation}\label{ssec:market}

The analysis of market concentration using mean-field stationary models has been initiated by \cite{Adlakha15}. In their general setting, the authors find as a sufficient condition for the existence of a stationary mean-field equilibrium the light-tail of the stationary distribution of players' size. In our model, we have seen that firms' size exhibit fat-tails, showing that the sufficient condition of \cite{Adlakha15} is not necessary. Besides, we introduce two indices to assess the effect of stability on market concentration in the presence of a continuum of firms. The first one is a limit version of the Herfindahl-Hirschman  index (HHI). The second one is a Gini index. The HHI index, used in market concentration analysis by most regulators, is defined as $H_n := \sum_{i=1}^n s_i^2$, where $n$ is the number of firms serving the market and $s_i$ is the market share of firm $i$. Highly fragmented market is obtained when each firm has the same market share and highly concentrated market is obtained with one firm holding all the market. In the first case, one has $H_n = 1/n$ while in the second case, $H_n = 1$. Note that in the $n$ firm case, a fully fragmented market corresponds to an average market share of $1/n$ and a zero variance in market share between firms, while a highly concentrated market corresponds also to an average market share of $1/n$, but with a variance $\mathbb V_n := 1/n - 1/n^2$. Thus, when the market is highly concentrated we note that $\mathbb V_n/\mathbb E[s_i] = 1 - 1/n$ admits a finite limit when $n$ gets large. Those remarks suggest to use as an index of market concentration the ratio between the variance of the firms' size and the square of its expectation at the stationary equilibrium. Another way to measure market concentration, already suggested in \cite{Hopenhayn92}, is to consider a Gini index. For a given quantile $q \in (0,1)$, define $x(q)$ the lowest $x$ such that $F(x) = \mathbb P(X_\infty\leq x) = q$. Then we define the Gini curve by $\bar Q(q) = \mathbb E[X_\infty | X_\infty \leq x(q)]/{\overline{Q}}^\star$~$\in (0,1)$, where we recall that ${\overline{Q}}^\star$ is the equilibrium average production across the regimes. Finally, the Gini index of market concentration is defined as 
\begin{align}
H := \int_{0}^1 \big| q -  \bar Q(q) \big|  dq.
\end{align}
The H-index measure deviation from uniform distribution of market shares as measured by the capacity held by firms at each level of quantiles. A fully fragmented market would yield a zero H-index, whereas a fully concentrated market served by a single monopolistic firm would induce an H-index of 1/2.

\begin{figure}
\begin{center}
\begin{tabular}{c c}
(a) & (b) \\
\includegraphics[width=0.45\textwidth]{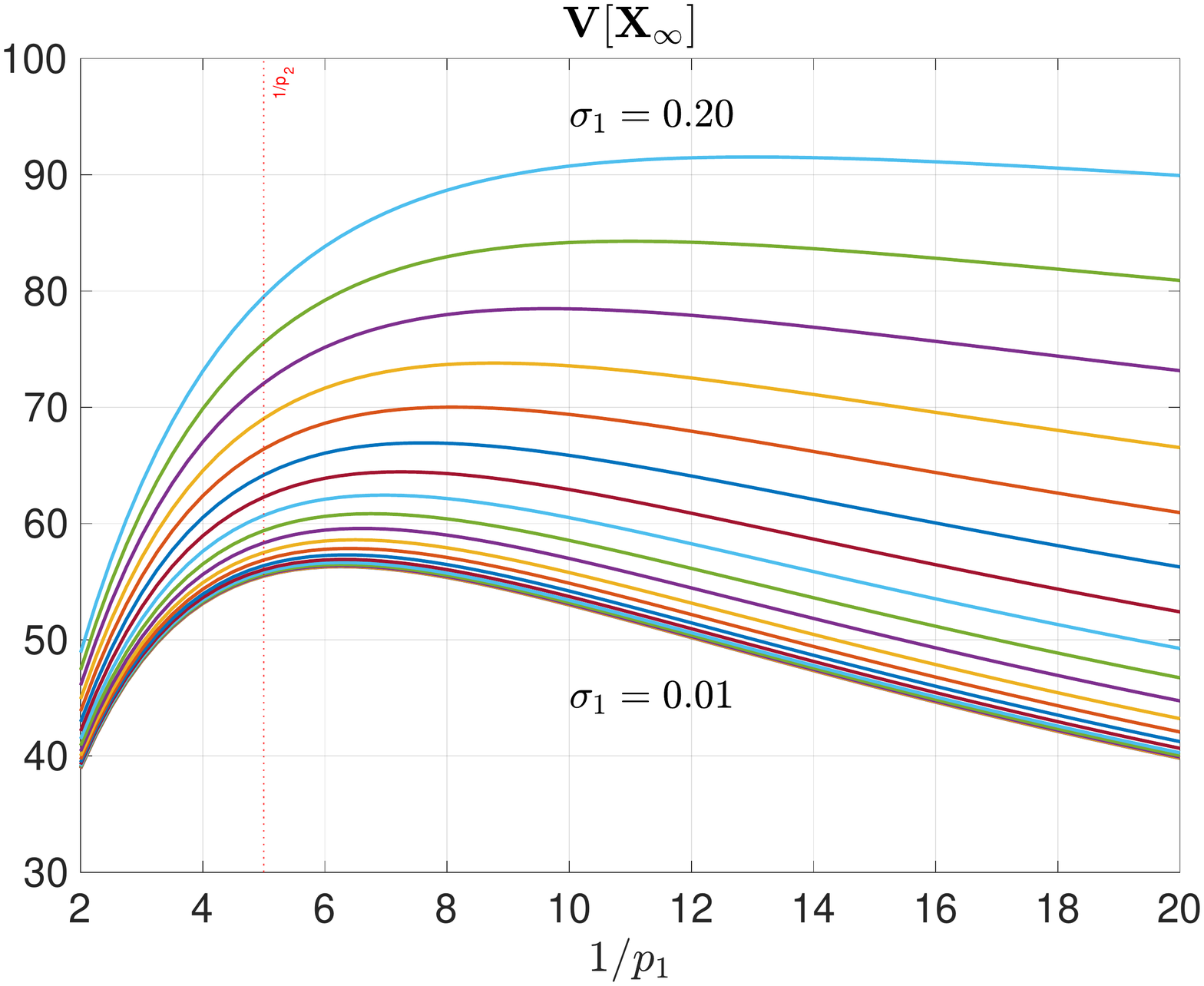} &
\includegraphics[width=0.45\textwidth]{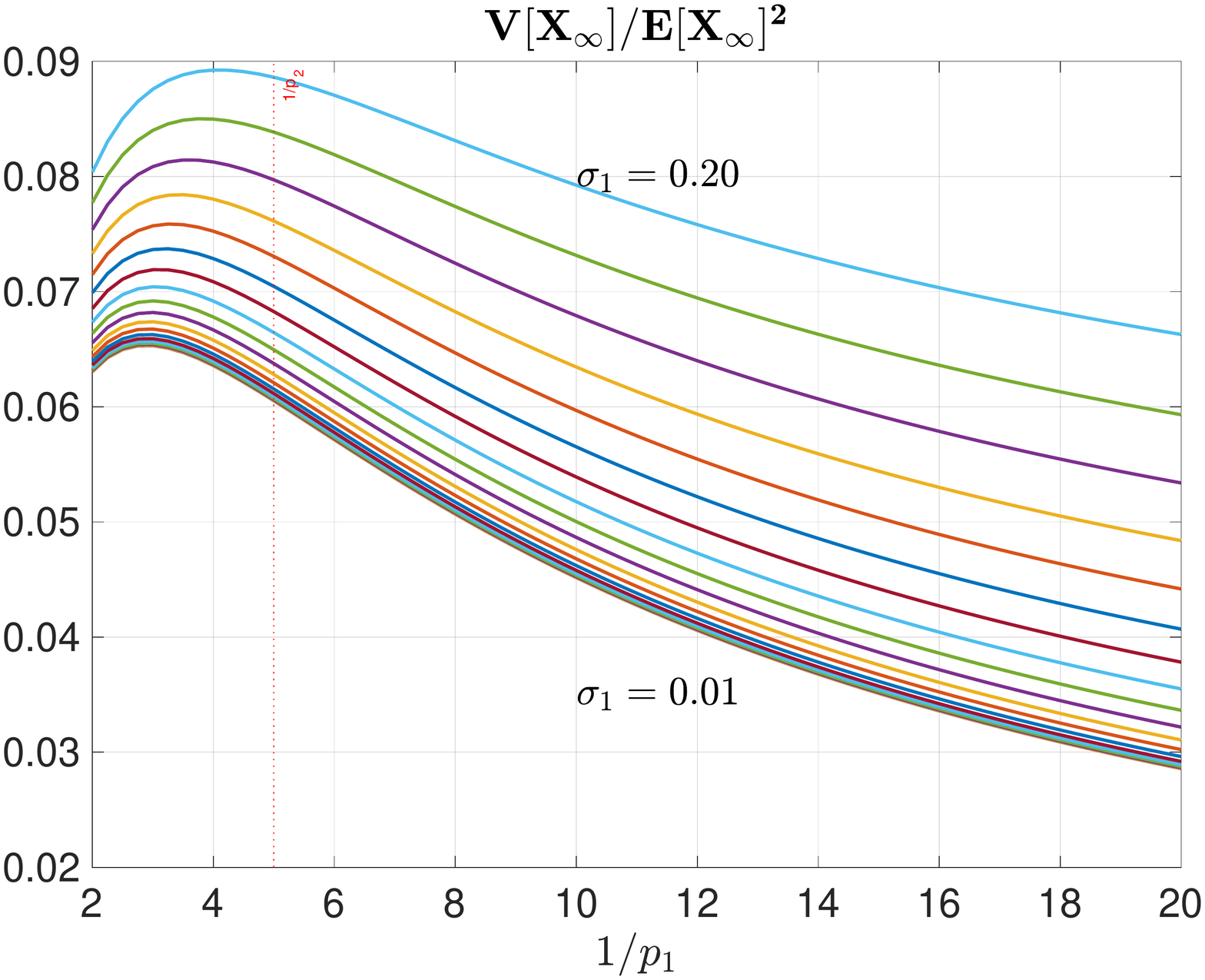}  \\
(c) & (d) \\
\includegraphics[width=0.45\textwidth]{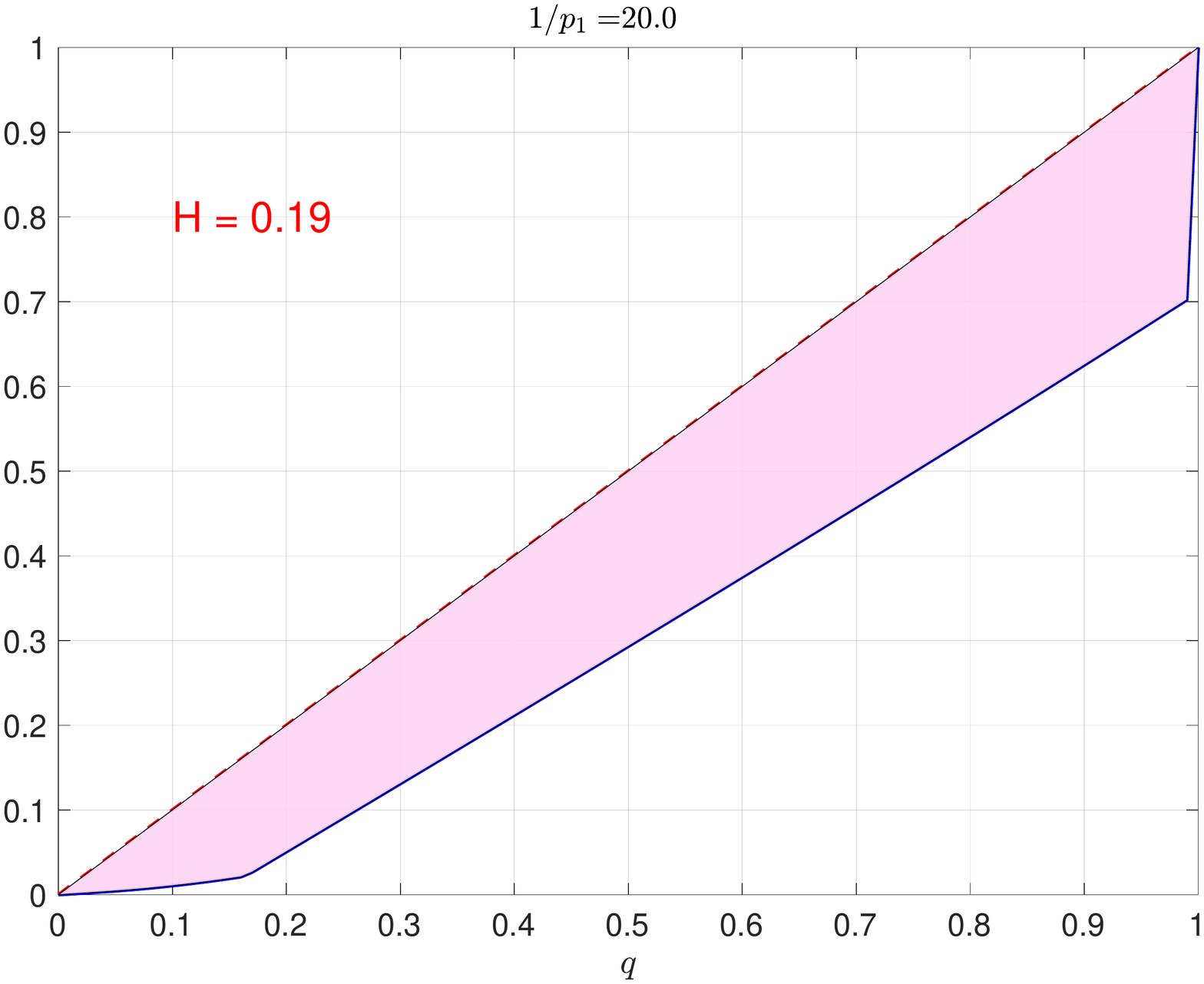} &
\includegraphics[width=0.45\textwidth]{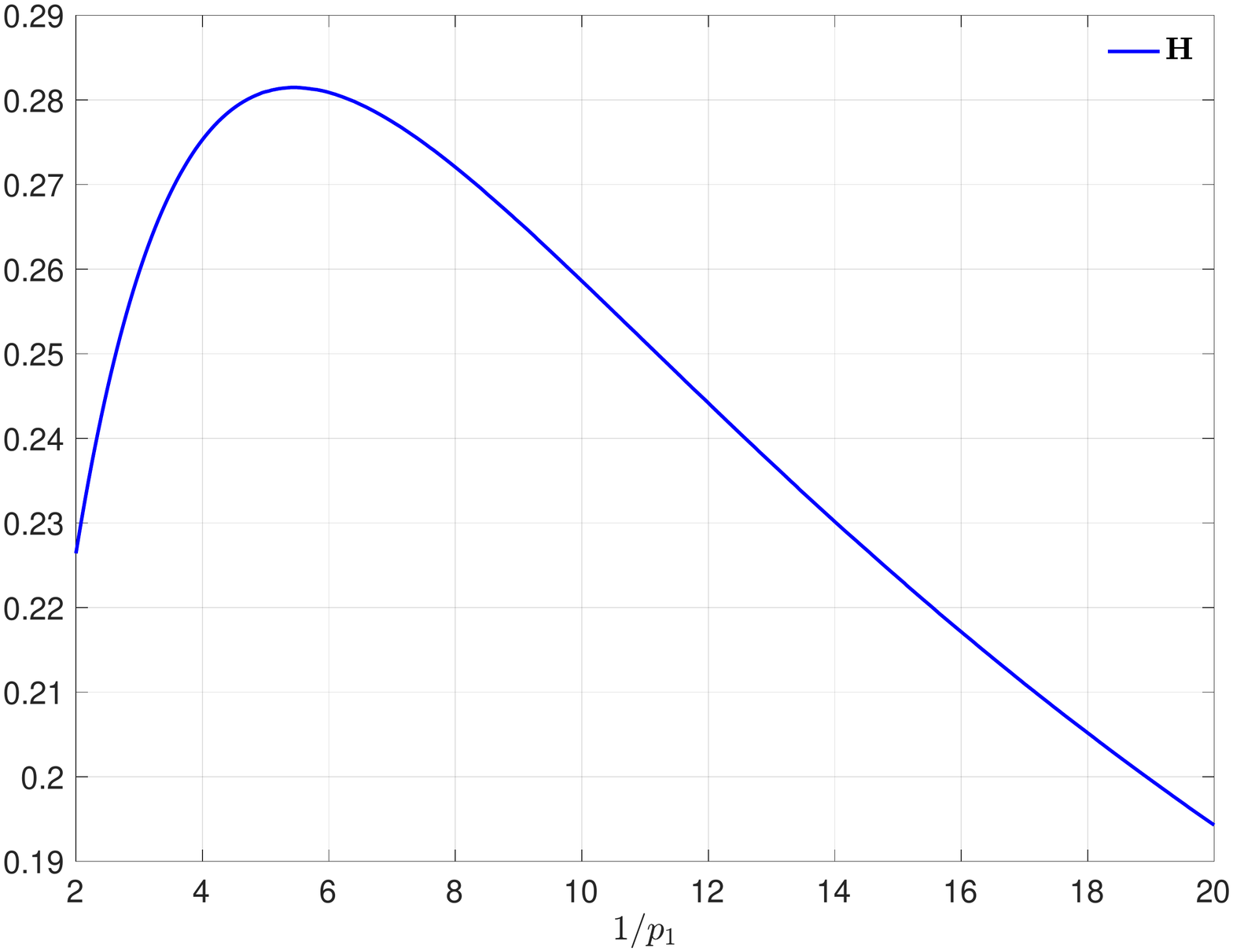} 
\end{tabular}
\caption{{\small (a) $\mathbb V[X_\infty]$ as a function of $1/p_1$; (b) $\mathbb V[X_\infty]/\mathbb E[X_\infty]$ as a function of $1/p_1$; (c)  $\bar Q(q)$ for $1/p_1 =20$ (blue curve), $H$ is the area of pink surface; (d) $H$ as a function of $1/p_1$. Parameters: $\delta=0.1$, $\rho=0.08$,  $\kappa=10$, $c=0.1$, $\phi_1=10$, $\phi_2=5$, $\zeta_1=\zeta_2=1$,  $\sigma_2=0.2$, $\alpha=0.5$, $p_2=1/5$.}}
\label{fig:frag}
\end{center}
\end{figure}

Figure~\ref{fig:frag} provides both the variance of the stationary process $X$ (see (a)) and the ratio of its variance with its expectation (see (b)) as a function of both the volatility $\sigma_1$ of State 1 and its persistence $1/p_1$. We observe that both are non-monotonic function of $1/p_1$. But when $1/p_1$ becomes large, both market concentration and variance decrease, meaning that market concentration reduces with more stable State 1. Figure~\ref{fig:frag}~(c) and~(d) provide the evaluation of the market concentration based on the Gini index. If a Gini index of $0.19$ is not a clear indication of market concentration, Figure~\ref{fig:frag}~(a) indicates that in that situation, less than 1 firm over a thousand holds $30\%$ of the total capacity. Besides, we observe that both indicators, HHI and Gini, present consistent results.  The Gini index reaches a peak close to $1/p_2$ and then decreases. Short periods of high prices compared to low prices tend to increase market concentration, while when high price periods become longer, market concentration decreases. Thus, longer period of stable economy tends to lower market concentration or, if we state this result in terms of instability, instability increases market concentration.

\subsection{Cost of crisis and benefit of sustainable growth}\label{ssec:crisis}

Consider the value $V^\star$ defined as
\begin{align*}
V^\star :=  \sum_{i=1}^2\int_0^{+\infty} V(x,i) p_\infty(dx,i).
\end{align*}

The value $V ^\star$ measure the firms expected profit at the stationary equilibrium. We want to analyze the effect of volatilities $\sigma_i$ and intensities of crisis $p_i$ on $V^\star$. Indeed, this point relates to the remark formulated by \cite{Lucas77} (pp.\ 25-31) on the cost of volatility growth. Using a simple model of intertemporal maximization of utility of a risk-averse representative agent, Lucas finds that excessively high risk-aversion would be necessary to justify the economic interest to hedge society against macroeconomic fluctuations. This result was exemplified using US post-world war II GDP. It was later reassessed in the macroeconomic literature, in particular in \cite{Epaulard03}, where the authors use more recent growth data and an endogenous growth model. Nevertheless, they still reach the conclusion of a small macroeconomic cost of growth rate fluctuations. Our model allows to investigate at a sectoral level the relative effects of Gaussian fluctuations compared to regime-switching from good state to bad state. Thus, we can measure how the existence of a persistent state of the economy with low selling prices affects the value of firms. We reduce the number of parameters characterizing a state to a vector $\nu_i := (\sigma_i,p_i,\phi_i)$ so that $\zeta_1=\zeta_2=1$. To make things comparable, we consider $V^\star(\nu_1,\nu_2)$ as a function of the two-state parameters $\nu_i=(\sigma_i,p_i,\phi_i)$. Then, we evaluate the elasticities of the stationary value  $V^\star$ w.r.t.\ the State 1 volatility~$\sigma_1$, intensity of switch~$p_1$ and the average level of price~$\phi_1$ at the point $(\nu_2,\nu_2)$. Having quasi-closed form expression of $V^\star$, we can compute the elasticities
\begin{align}\label{eq:elas}
\chi_{\sigma_1}(\nu_2) :=& - \frac{\sigma_2}{V^\star(\nu_2,\nu_2)} \frac{\partial V^\star}{\partial \sigma_1}(\nu_2,\nu_2), \quad
\chi_{p_1}(\nu_2) := - \frac{p_2}{V^\star(\nu_2,\nu_2)} \frac{\partial V^\star}{\partial p_1}(\nu_2,\nu_2), \\
\chi_{\phi_1}(\nu_2) :=&  \frac{\phi_2}{V^\star(\nu_2)} \frac{\partial V^\star}{\partial \phi_1}(\nu_2,\nu_2), \nonumber
\end{align}
upon assuming that State 1 is slightly better than the State 2 (that is, $\phi_1$ slightly larger than $\phi_2$, $\sigma_1$ slightly lower than $\sigma_2$, and $p_1$ slightly lower than $p_2$). For instance, taking $\delta=1/10$, $\rho=0.08$, $\alpha=1/2$, $c=0.1$, $\kappa=10$ and $\nu_2 = (0.2,1/10,15)$, we find
\begin{align}
\chi_{\sigma_1}(\nu_2) = 0.08, \quad
\chi_{p_1}(\nu_2) = 0, \quad
\chi_{\phi_1}(\nu_2) =  1.06.
\end{align}

\begin{table}[tbh!]
\center
\begin{tabular}{l l | c  c }
$\sigma_2$ & $\phi_2$ & $\chi_{\sigma_1}(\nu_2)$ &  $\chi_{\phi_1}(\nu_2)$ \\
& & & \\ \hline
& & & \\ 
$0.1$ & $10$ & $0.004$ & $1.1$ \\
$0.1$ & $15$ & $0.004$ & $1.07$ \\
$0.2$ & $10$ & $0.08$ & $1.09$ \\
$0.2$ & $15$ & $0.08$ & $1.06$ \\
$0.3$ & $10$ & $0.55$ & $1.08$ \\
$0.3$ & $15$ & $0.55$ & $1.05$ \\
\end{tabular}
\caption{Elasticities of $V^\star$ at $\nu_2 = (\sigma_2,p_2,\phi_2)$ w.r.t $\sigma_1$ and $\phi_1$. Other parameters value: $\delta=1/10$, $\rho=0.08$, $\kappa=10$, $c=0.1$, $\zeta_1=\zeta_2=1$, $p_2=1/10$.}\label{tab:vstar}
\end{table}

Note that $\partial_{p_1} V^\star(\nu_2,\nu_2) = 0$, because when taking the partial derivative w.r.t.\ $p_1$, and calculating it at $(\nu_2,\nu_2)$, the parameters $(\sigma_i,\phi_i)$, $i=1,2$, are the same across the regimes. Increasing the volatility of State 1 by 1\% decreases the stationary value of the firm by 0.08\%, but reducing the price by 1\% cuts the value $V^\star$ by nearly the same amount. Table~\ref{tab:vstar} provides a few other values of elasticities when $\sigma_2$ and $\phi_2$ vary. We observe a quasi-constant elasticity of $V^\star$ w.r.t.\ $\phi_1$ close to one and a lower elasticity w.r.t.\ the volatility $\sigma_2$, but sharply increasing. It means that the marginal cost of volatility is increasing while the marginal gain from stability is decreasing.

\section{Equilibrium Construction and Proof of Theorem \ref{thm:Equilibrium}}
\label{sec:Equiconstr}

In this section, we provide a constructive proof of Theorem \ref{thm:Equilibrium} according to the following three-step recipe (which will require the statement and proof of auxiliary intermediate results):
\begin{enumerate}
    \item \textbf{First step} (Section \ref{ssec-step1}): Given $Q=(Q_1,Q_2) \in \rr^2_+$, we determine $I^{\star}(Q)$ such that 
    \begin{equation}
    \label{eq-goal-step1}
    J_{(x,i)}(I^{\star}(Q),Q) = \sup_{I \in \mathcal{A}} J_{(x,i)}(I,Q) =: V^Q(x,i).
    \end{equation}
    \item \textbf{Second step} (Section \ref{ssec-step2}): Given $Q=(Q_1,Q_2) \in \rr^2_+$ and $(I^{\star}_t(Q))_{t\geq 0}$ from the first step, we determine the stationary distribution for $(X^{I^{\star}(Q)}_t,\eps_t)_{t \geq 0}$, denoted by
    \begin{equation}
    \label{eq-goal-step2}
    \big(p_\infty^Q(dx,i)\big)_{i=1,2}.
    \end{equation}
    \item \textbf{Third step} (Section \ref{ssec-step3}): we look for $Q^{\star}=(Q^{\star}_1,Q^{\star}_2) \in \rr^2_+$ solving the fixed-point problem
    \begin{equation}
    \label{eq-goal-step3}
    Q_i^{\star} = \frac{1}{\pi_i} \int_0^\infty x \, p^{Q^{\star}}_\infty(dx,i), \qquad i=1,2,
    \end{equation}
    with $(p_\infty^Q(dx,i))_{i=1,2}$ from the second step. By construction, the pair $(I^{\star}({Q^{\star}}),Q^{\star})$ is a MFE. 
\end{enumerate}


\subsection{First step: Solving the singular control problem}
\label{ssec-step1}

Throughout this section, we let
\begin{equation*}
Q=(Q_1,Q_2) \in \rr^2_+    
\end{equation*}
be given and fixed and solve the singular stochastic control problem \eqref{eq-goal-step1}. To simplify notation, we will omit the dependence on $Q$ and write, for example, $V(x,i)$ instead of $V^Q(x,i)$, $J_{(x,i)}(I)$ instead of $J_{(x,i)}(I,Q)$, and $I^{\star}$ instead of $I^{\star}(Q)$. 

First, in Section \ref{sssec-step1a} we prove a preliminary verification theorem for the singular control problem. This is then exploited in Section \ref{ref:OS-problem}, where we characterize $V$ in terms of the value function $v$ of a suitable optimal stopping problem with regime switching. Finally, in Section \ref{sssec-step1b}, we determine (semi-)closed expression for $v$ (hence, for $V$) and a system of nonlinear algebraic equations solved by the endogenous thresholds triggering the optimal stopping rule and the optimal investment strategy.

\subsubsection{A preliminary verification theorem} 
\label{sssec-step1a}

Recall that we are dealing with the singular stochastic control problem with regime-switching (cf.\ \eqref{eq-goal-step1})
\begin{equation}
\label{eq-pb-SC}
V(x,i) = \sup_{I \in \mathcal{A}} J_{(x,i)}(I), \qquad (x,i) \in \rr_+ \times \{1,2\},
\end{equation}
with $J_{(x,i)}(I)=J_{(x,i)}(I,Q)$ as in \eqref{eq-J}. We start by deducing the dynamic programming equation that we expect to be associated to problem \eqref{eq-pb-SC}. In the sequel, we denote by $\bar 0$ the control which is identically zero. Heuristically, the dynamic programming principle suggests that, for a small time step $\Delta t$, 
\begin{equation*}
V(x,i) \geq \eee_{(x,i)} \Bigg[e^{-\rho \Delta t} V\big(X^{\bar 0}_{\Delta t},\eps_{\Delta t}\big) + \int_0^{\Delta t} e^{-\rho t} \Big(X^{\bar 0}_t \eta_{\eps_t} - c \big(X^{\bar 0}_t\big)^2 \Big) dt\Bigg],
\end{equation*}
from which, applying Dynkin's formula to  $\eee_{(x,i)}[e^{-\rho \Delta t} V(X^{\bar 0}_{\Delta t},\eps_{\Delta t})]$, dividing by $\Delta t$, letting $\Delta t \to 0$, and assuming that the mean-value and dominated convergence theorems hold, we find
\begin{equation}
\label{eq-DPP1}
(\mathcal{L}-\rho)V(x,i) + x\eta_i - cx^2 \leq 0, \qquad (x,i) \in \rr_+ \times \{1,2\}.
\end{equation}
Here, the infinitesimal generator $\mathcal{L}$ is defined as
\begin{equation}
\label{eq-mathcalL}
\mathcal{L}w(x,i) := \frac12 \sigma_i^2 x^2 w''(x,i) - \delta x w'(x,i) + \sum_{j=1}^2 P_{ij} w(x,j), \qquad w(\cdot,i) \in C^2(\rr_+), \, i \in \{1,2\}.
\end{equation}
In \eqref{eq-mathcalL}, $P$ denotes the generator matrix of the Markov chain $\eps$, see \eqref{eq-Pi}, and differentiation is always meant with respect to $x$. On the other hand, investing $h>0$ at time $t=0$ and then following an optimal control rule (if one does exists) gives
\begin{equation*}
V(x,i) \geq V\Big(x +  x (e^{ h}-1), i\Big) -  \kappa  x (e^{ h}-1 ) \sim V(x +  h x, i) -  \kappa hx, 
\end{equation*}
and hence
\begin{equation*}
\frac{V(x +  h x, i) - V(x,i) }{h x} \leq  \kappa,
\end{equation*}
which, letting $h \to 0$, suggests that
\begin{equation}
\label{eq-DPP2}
 V'(x,i) \leq  \kappa, \qquad (x,i) \in \rr_+ \times \{1,2\}.
\end{equation}
Since one of the two possibilities (intervening or not intervening) is optimal, one of the two inequalities \eqref{eq-DPP1} and \eqref{eq-DPP2} is indeed an equality. Overall, we get the following candidate equation for $V$:
\begin{equation}
\label{eq-VI-SC}
\max \Big\{ (\mathcal{L}-\rho)w(x,i) + x\eta_i - cx^2, \,\,  w'(x,i) -  \kappa  \Big\} =0, \qquad (x,i) \in \rr_+ \times \{1,2\}.
\end{equation}
This is a system of ordinary differential equations (ODEs), due to the transition amongst the regimes, with gradient constraints. 

Because of the structure of the problem, we expect that the representative company does not intervene until its production falls below a certain critical level, depending on the state of the economy. In other words, we expect the company's no-action region to be in the form
\begin{equation}
\label{eq-continuation-SC}
\mathcal{NA}:=\big\{(x,i) \in \rr^2_+ \times \{1,2\}:\, w'(x,i)< \kappa \big\}=\big\{(x,i)\in \rr_+ \times \{1,2\}:\, x>a_i\big\},
\end{equation}
for some constants $a_i$, $i \in \{1,2\}$ to be determined. Notice that, under the additional assumption \eqref{eq-continuation-SC}, equation \eqref{eq-VI-SC} is in fact equivalent to the free-boundary problem
\begin{equation}
\label{eq-fbp-SC}
\begin{cases}
(\mathcal{L}-\rho)w(x,i) + \eta_ix - cx^2 \leq 0, & x < a_i, \\
(\mathcal{L}-\rho)w(x,i) + \eta_ix - cx^2 =0, & x \geq a_i, \\
w'(x,i)-\kappa=0, & x \leq a_i, \\
w'(x,i)-\kappa< 0, & x > a_i.
\end{cases}
\end{equation}
Based on  \eqref{eq-fbp-SC}, the following verification theorem for problem \eqref{eq-pb-SC} holds. 

\begin{prop}
\label{prop-verif-SC}
Let $w: \rr_+\times \{1,2\} \to \rr$ such that:
\begin{itemize}
    \item[(i)] for $i \in \{1,2\}$, $w(\cdot, i) \in C^2(\rr_+)$ and there exists $K>0$ such that $|w(x,i)|\leq K(1+|x|^2)$ for any $x \in \rr_+$; 
    \item[(ii)] there exists $(a_1, a_2) \in \rr^2_+$ such that \eqref{eq-continuation-SC} holds; 
    \item[(iii)] $w$ is a solution to the free-boundary problem \eqref{eq-fbp-SC}.
\end{itemize}
Then, the value function $V$ of the singular control problem \eqref{eq-pb-SC} identifies with $w$, $V \equiv w$, and the optimal control is given by
\begin{equation}
\label{eq-optctrl-SC}
I^{\star}_t = 0 \,\, \vee \,\,   \sup_{0 \leq s \leq t} \bigg( \ln(a_{\eps_s}/x) + \int_0^s (\delta +\frac12 \sigma_{\eps_u}^2) du - \int_0^s \sigma_{\eps_u} dB_u \bigg), \qquad t \geq 0, \qquad I^{\star}_{0-}=0.
\end{equation}
\end{prop}

\begin{proof}
Postponed to Appendix \ref{sec-appendix}. 
\end{proof}

Notice that the optimal control $I^{\star}$ in \eqref{eq-optctrl-SC} is such that $(X^{I^{\star}}, I^{\star}, \eps)$ solves a Skorokhod reflection problem at the regime-dependent boundary $a_{\eps}$ (cf.\ \cite{BudLiu}). In particular, $\mathbb{P}$-a.s.\ for all $t \geq 0$:
\begin{equation}
\label{eq-Skrefl}
X^{I^{\star}}_t \geq a_{\eps_t}, \qquad I^{\star}_t = \int_{[0,t]} \mathds{1}_{\{X^{I^{\star}}_{s-} \leq a_{\eps_{s}}\}} dI^{\star}_s, \qquad \int_0^{\Delta I^{\star}_t} \mathds{1}_{(X^{I^{\star}}_{t-} + z, \eps_{t}) \in \mathcal{NA}\}} dz =0.
\end{equation}
That is, $I^{\star}$ keeps $(X^{I^{\star}}, \eps)$ for all times in the closure of the no-action region $\{(x,i)\in \rr_+ \times \{1,2\}:\, V'(x,i) < \kappa \}$, and it acts only as much as it is necessary to prevent that the state-process leaves such a portion of the state-space.  

Solving problem \eqref{eq-fbp-SC} by a guess-and-verify approach (i.e., considering suitable parametric candidate solutions to \eqref{eq-VI-SC} and selecting the ``optimal'' parameters by imposing appropriate regularity conditions, the so-called smooth-fit and smooth-pasting conditions) is possible but challenging in the present context. As a matter of fact, the underlying Markov chain makes the system of constrained ODEs in \eqref{eq-VI-SC} interconnected, with the result that it becomes very difficult to show existence and uniqueness of the solution to the highly nonlinear (and unhandy) smooth-fit and smooth-pasting equations. We therefore adopt a different approach here, already employed in \cite{FeRo}: We introduce an optimal stopping problem with regime switching; We prove existence of thresholds triggering its optimal stopping rule and regularity of its value function; finally, by means of Proposition \ref{prop-verif-SC}, we verify that a suitable integral of the stopping problem's value function identifies with the value function of the singular control problem \eqref{eq-pb-SC}.

\subsubsection{A related optimal stopping problem and its relation to the singular control problem}
\label{ref:OS-problem}

Let us start by defining
\begin{equation}
\label{eq-pb-OS}
v(x,i) := \inf_{\tau \geq 0} \, \bar{\eee}_{(x,i)} \bigg[ \int_0^\tau e^{-(\rho + \delta)t}(\eta_{\eps_t} - 2c \bar X_t) dt + \kappa e^{-(\rho+\delta)\tau} \bigg],\quad (x,i) \in \rr_+ \times \{1,2\},
\end{equation}
where $\tau$ has to be chosen in the class of $\mathbbm{F}$-stopping times and $\bar X$ satisfies
\begin{equation}
\label{eq-Xbar}
d\bar X_t = - (\delta - \sigma^2_{\eps_t}) \bar X_t dt + \sigma_{\eps_t} \bar X_t dB_t, \qquad \bar X_0=x >0.
\end{equation}
In \eqref{eq-pb-OS} above, $\bar{\eee}_{(x,i)}$ denotes the expectation under $\bar{\mathbb{P}}_{(x,i)}[\,\cdot\,]:=\mathbb{P}[\cdot\,|\,\bar X_0=x, \eps_0=i]$. Furthermore, for future frequent use, we note that the following second-order differential operator is the infinitesimal generator associated to $(\bar X, \eps)$:
\begin{equation}
\label{eq-mathcalLbar}
\bar{\mathcal{L}}w(x,i) = \frac12 \sigma_i^2 x^2 w''(x,i) - (\delta-\sigma_i^2) x w'(x,i) + \sum_{j=1}^2 P_{ij} w(x,j), \qquad w(\cdot,i) \in C^2(\rr_+), \, i \in \{1,2\},
\end{equation}
with $P$ as in \eqref{eq-Pi}.  

As it is customary in optimal stopping theory (see \cite{PeskShir}), we introduce the continuation and stopping regions of Problem \eqref{eq-pb-OS}
\begin{equation}
\label{eq-continuation-OS}
\mathcal{C}:=\{(x,i)\in \rr_+ \times \{1,2\}:\, v(x,i)<\kappa \}, \quad \mathcal{S}:=\{(x,i)\in \rr_+ \times \{1,2\}:\, v(x,i) \geq \kappa \},
\end{equation}
as well as their $i$-sections: 
$$\mathcal{C}_i:=\{x\in \rr_+:\, (x,i) \in \mathcal{C}\}, \quad \mathcal{S}_i:=\{x\in \rr_+:\, (x,i) \in \mathcal{S}\}.$$

Then, the next result holds true.
\begin{prop}
\label{prop-verif-OS}
The value function $v$ of the optimal stopping problem \eqref{eq-pb-OS} satisfies the following properties: 
\begin{itemize}
    \item[(i)] there exists $(a_1, a_2) \in \rr^2_+$ such that 
    $$\mathcal{C}_i=\{x\in \rr_+:\, x < a_i\}, \quad \mathcal{S}_i=\{x\in \rr_+:\, x \geq a_i\}.$$
    \item[(ii)] 
    $v(\cdot, i) \in C^1(\rr_+)$, $v(\cdot, i) \in C^2(\mathcal{C}_i)$, and $v''(\cdot, i) \in L^\infty_{loc}(\rr_+)$, for $i \in \{1,2\}$, and satisfies
\begin{equation}
\label{eq-fbp-OS}
\begin{cases}
(\bar{\mathcal{L}}-(\rho+\delta))v(x,i) + \eta_i - 2cx \geq 0, & x < a_i, \\
(\bar{\mathcal{L}}-(\rho+\delta))v(x,i) + \eta_i - 2cx =0, & x > a_i, \\
\kappa- v(x,i) \geq 0, & x < a_i, \\
\kappa- v(x,i) =0, & x \geq a_i. \\
\end{cases}
\end{equation}
   \item[(iii)] $\tau^{\star}(x,i):= \inf \{ t \geq 0 : \bar X_t \leq a_{\eps_t}\},\,\, \bar{\mathbb{P}}_{(x,i)}$-a.s., is optimal for \eqref{eq-pb-OS}.
\end{itemize}
\end{prop}

\begin{proof}
Postponed to Appendix \ref{sec-appendix}.
\end{proof}

The next proposition finally links the optimal stopping problem \eqref{eq-pb-OS} to the singular control problem \eqref{eq-pb-SC}. 
\begin{prop}
\label{prop-conclus-step1}
Let $v(x,i)$ and $a_i$ be as in Proposition \ref{prop-verif-OS}, and $k_1$ and $k_2$ be the unique solutions to the linear system 
\begin{equation}
\label{eq-k}
\begin{cases}
-(\rho+p_1)k_1 + p_1k_2 + \Big(\eta_1 -   \delta a_1 \kappa\Big) a_1 - ca_1^2 - p_1 \displaystyle \int_{a_1}^{a_2} v(y,2) dy=0, \\[0.4cm]
p_2k_1 -(\rho+p_2)k_2 + \Big(\eta_2 -  \delta a_2 \kappa \Big) a_2 - ca_2^2 + p_2 \displaystyle \int_{a_1}^{a_2} v(y,1) dy=0.
\end{cases}
\end{equation}
Then, the value function $V$ of the singular control problem \eqref{eq-pb-SC} satisfies
\begin{equation}
\label{eq-link-Vv}
V(x,i) = k_i + \int_{a_i}^x v(y,i) dy, \qquad\quad (x,i) \in \rr_+ \times \{1,2\}.
\end{equation}
Furthermore, the optimal control $I^{\star}$ is given by \eqref{eq-optctrl-SC}.
\end{prop}

\begin{proof}
Postponed to Appendix \ref{sec-appendix}.
\end{proof}


\subsubsection{Semi-closed expressions for $v$ and equations for the free boundaries}
\label{sssec-step1b}

Due to Proposition \ref{prop-conclus-step1}, $v$ and $a_i$, $i=1,2$, solve Problem \eqref{eq-fbp-OS}. We then here focus on solving \eqref{eq-fbp-OS} and, without loss of generality, we assume
\begin{equation}
\label{eq-order-a}
a_1 \leq a_2.
\end{equation}

The other case can be treated via completely analogous arguments and it is therefore omitted in the interest of brevity. We stress once more that solving \eqref{eq-fbp-OS} and imposing the regularity prescribed in Proposition \ref{prop-verif-OS}-(ii) we do obtain the value function of the optimal stopping problem (and not candidate values), as well as equations that are necessarily satisfied by the free boundaries $a_1$ and $a_2$. 

We solve Problem \eqref{eq-fbp-OS} by considering separately the three intervals $(0,a_1)$, $(a_1,a_2)$, $(a_2,\infty)$. The case $x \in (0,a_1)$ is trivial: By \eqref{eq-fbp-OS} we have 
\begin{equation*}
v(x,1) = \kappa= v(x,2).    
\end{equation*}
When $x \in (a_1,a_2)$, functions $v(x,1)$ and $v(x,2)$ satisfy
\begin{equation*}
\begin{cases}
(\bar{\mathcal{L}}-(\rho+\delta))v(x,1) + \eta_1 - 2cx = 0, \\
v(x,2) = \kappa,
\end{cases}    
\end{equation*}
that is, by the definition of $\mathcal{L}$ in \eqref{eq-mathcalL},
\begin{equation*}
\begin{cases}
\frac12 \sigma_1^2 x^2 v''(x,1) - (\delta-\sigma_1^2) x v'(x,1) - (\rho + \delta + p_1) v(x,1) = - \eta_1 + 2cx   - p_1\kappa, \\
v(x,2) = \kappa.
\end{cases}    
\end{equation*}
The general solution to the first equation is 
\begin{equation*}
v(x,1) = A \Big(\frac{x}{a_1}\Big)^{\gamma_1^+} + B \Big(\frac{x}{a_1}\Big)^{\gamma_1^-} + \hat{v}(x,1),  
\end{equation*}
where $A,B \in \rr$ are free parameters (to be specified later) and where $\gamma_1^+,\gamma_1^-$ are the solutions to
\begin{equation*}
\frac12 \sigma_1^2\gamma(\gamma - 1) - (\delta-\sigma_1^2) \gamma - (\rho+\delta+p_1) = 0;
\end{equation*}
that is,
\begin{equation*}
\gamma_{1}^{\pm} = \frac{(\delta-\frac{\sigma_1^2}{2}) \pm \sqrt{(\delta-\frac{\sigma_1^2}{2})^2 + 2 \sigma_1^2(\rho+\delta+p_1)}}{\sigma_1^2} \qquad\qquad(\gamma_1^- <0<\gamma_1^+).  
\end{equation*}
As for $\hat{v}(x,1)$, by looking for an affine particular solution, we find
\begin{equation*}
\hat{v}(x,1) = C_1 x + D_1,
\end{equation*}
where $C_1,D_1$ are given by (where Assumption \ref{ass:assrho} guarantees that $C_1$ is finite)
\begin{equation*}
C_1= - \frac{2c}{\rho + 2\delta + p_1 - \sigma^2_1}, \qquad\qquad 
 D_1= \frac{\eta_1 + p_1\kappa}{ \rho + \delta + p_1}.
\end{equation*}
Hence, in $(a_1,a_2)$ we have
\begin{equation*}
\begin{cases}
v(x,1) = A \big(\frac{x}{a_1}\big)^{\gamma_1^+} + B \big(\frac{x}{a_1}\big)^{\gamma_1^-} + C_1x + D_1, \\
v(x,2) = \kappa.
\end{cases}    
\end{equation*}
Finally, in $x \in (a_2, \infty)$ functions $v(x,1)$ and $v(x,2)$ satisfy
\begin{equation*}
\begin{cases}
(\bar{\mathcal{L}}-(\rho+\delta))v(x,1) + \eta_1 - 2cx = 0, \\
(\bar{\mathcal{L}}-(\rho+\delta))v(x,2) + \eta_2 - 2cx = 0,
\end{cases}    
\end{equation*}
that is, by the definition of $\bar{\mathcal{L}}$ in \eqref{eq-mathcalLbar},
\begin{equation*}
\begin{cases}
\frac12 \sigma_1^2 x^2 v''(x,1) - (\delta-\sigma_1^2) x v'(x,1) - (\rho + \delta + p_1) v(x,1) + p_1 v(x,2) = - \eta_1 + 2cx, \\
\frac12 \sigma_2^2 x^2 v''(x,2) - (\delta-\sigma_2^2) x v'(x,2) - (\rho + \delta + p_2) v(x,2) + p_2 v(x,1) = - \eta_2 + 2cx.
\end{cases}    
\end{equation*}
The general solution is given by 
\begin{equation*}
\begin{cases}
v(x,1) = \sum_{i=1}^4 M_i \big(\frac{x}{a_2}\big)^{\lambda_i} + \tilde{v}(x,1),\\
v(x,2) = \sum_{i=1}^4 \bar M_i \big(\frac{x}{a_2}\big)^{\lambda_i} + \tilde{v}(x,2),
\end{cases}    
\end{equation*}
where $M_i \in \rr$ are free parameters, 
\begin{equation*}
\bar M_i = - \frac{G_1(\lambda_i)}{p_1}M_i =: - G_{1i} M_i, \qquad\quad i \in \{1,2,3,4\},
\end{equation*}
with $\lambda_i$ being the unique real roots to the equation 
\begin{equation}
\label{eq:Gbar}
\bar G (\lambda) := G_1(\lambda)G_2(\lambda) - p_1p_2 = 0,
\end{equation}
where
\begin{equation*}
G_i(\lambda) = \frac12 \sigma_i^2\lambda(\lambda - 1) - (\delta-\sigma_i^2) \lambda - (\rho+\delta+p_i), \qquad\quad i \in \{1,2\}.
\end{equation*}
To see that \eqref{eq:Gbar} admits four distinct real roots $\lambda_1 < \lambda_2 < 0 < \lambda_3 < \lambda_4$, it suffices to notice that: (i) $\bar G(\gamma_1^\pm) = -p_1p_2 <0$, with $\gamma_1^-<0<\gamma_1^+$; (ii) $\bar G(0)>0$; (iii) $\bar G(+\infty) = \bar G(-\infty) = +\infty$.

As for $\tilde{v}(x,i)$, by looking again for affine solutions, one gets
\begin{equation*}
\tilde{v}(x,i) = L_i x + R_i, \qquad\quad i \in \{1,2\},
\end{equation*}
where $L_i,R_i$ are solutions to the linear systems
\begin{equation*}
\begin{bmatrix}
\rho + 2\delta + p_1 - \sigma^2_1 & -p_1\\
-p_2 & \rho + 2\delta + p_2 - \sigma^2_2
\end{bmatrix}
\begin{bmatrix}
L_1\\
L_2
\end{bmatrix}
=
\begin{bmatrix}
-2c\\
-2c
\end{bmatrix}
, \qquad
\begin{bmatrix}
\rho + \delta + p_1 & -p_1\\
-p_2 & \rho + \delta + p_2 
\end{bmatrix}
\begin{bmatrix}
R_1\\
R_2
\end{bmatrix}
=
\begin{bmatrix}
\eta_1\\
\eta_2
\end{bmatrix}
.
\end{equation*}
Furthermore, to guarantee a linear growth for $v(x,i)$, we take $M_3=M_4=0$. Hence, in $(a_2, \infty)$ we have
\begin{equation*}
\begin{cases}
v(x,1) = M_1 \big(\frac{x}{a_2}\big)^{\lambda_1} + M_2 \big(\frac{x}{a_2}\big)^{\lambda_2} + L_1 x + R_1,\\[0.2cm]
v(x,2) = - M_1 G_{11} \big(\frac{x}{a_2}\big)^{\lambda_1} - M_2 G_{12} \big(\frac{x}{a_2}\big)^{\lambda_2} + L_2 x + R_2.
\end{cases}    
\end{equation*}
To determine the value of the parameters $A,B,M_1,M_2,a_1,a_2$, we impose that $v(\cdot,i) \in C^1(\rr_+)$ for $i \in \{1,2\}$ (cf.\ Proposition \ref{prop-verif-OS}-(ii)):

\begin{equation*}
\begin{cases}
v(a_{1}-,1) = v(a_{1}+,1), \\
v'(a_{1}-,1) = v'(a_{1}+,1), \\
v(a_{2}-,2) = v(a_{2}+,2), \\
v'(a_{2}-,2) = v'(a_{2}+,2), \\
v(a_{2}-,1) = v(a_{2}+,1), \\
v'(a_{2}-,1) = v'(a_{2}+,1),
\end{cases}
\end{equation*}
where, as usual, for a function $f:\rr \times \{1,2\} \to \rr$, $f(x_o \pm, i)$ represents the right/left limit at a given point $x_o$.
The latter system of conditions leads to


\begin{equation}
\label{eq-smoothfit-stopping}
\begin{cases}
\kappa = A + B + C_1 a_1 + D_1, \\[0.15cm]
0 = A \gamma_1^+ + B\gamma_1^- + Ca_1, \\[0.15cm]
\kappa = - M_1 G_{11} - M_2 G_{12} + L_2a_2 + R_2,  \quad j=1,2\\[0.15cm]
0 = - M_1 \lambda_1G_{11} - M_2 \lambda_2 G_{12} + L_2 a_2, \\[0.15cm]
A \big(\frac{a_2}{a_1}\big)^{\gamma_1^+} + B \big(\frac{a_2}{a_1}\big)^{\gamma_1^-} + C_1 a_2 + D_1 = M_1 + M_2 + L_1a_2 + R_1, \\[0.15cm]
 A \gamma_1^+ \big(\frac{a_2}{a_1}\big)^{\gamma_1^+} + B\gamma_1^- \big(\frac{a_2}{a_1}\big)^{\gamma_1^-} + C_1 a_2 = M_1 \lambda_1 + M_2 \lambda_2 + L_1 a_2. \\
\end{cases}
\end{equation}

By considering the pairs \eqref{eq-smoothfit-stopping}(i-ii) and \eqref{eq-smoothfit-stopping}(iii-iv), one can express $A,B$ and $M_1,M_2$ as functions of the unknown $a_1,a_2$:
\begin{align*}
A &=  \bigg(\frac{(\kappa-D_1) \gamma_1^-}{\gamma_1^- -\gamma_1^+} \bigg) + \bigg(\frac{C_1 (1-\gamma_1^-)}{\gamma_1^- -\gamma_1^+}\bigg) a_1  =: c_{11} + c_{12} a_1,
\\
B &=  \bigg(- \frac{(\kappa-D_1) \gamma_1^+}{\gamma_1^- -\gamma_1^+}\bigg) + \bigg(- \frac{C_1 (1-\gamma_1^+)}{\gamma_1^- -\gamma_1^+} \bigg)a_1 =: c_{21} + c_{22} a_1, \\
M_1 &=   \bigg(-  \frac{(\kappa-R_2) \lambda_2}{ G_{11}(\lambda_2 -\lambda_1)}\bigg) + \bigg(-  \frac{L_2 (1-\lambda_2)}{G_{11}(\lambda_2 -\lambda_1)}\bigg)a_2 =: d_{11} + d_{12} a_2,
\\
M_2 &=  \bigg(  \frac{(\kappa-R_2) \lambda_1}{G_{12}(\lambda_2 -\lambda_1)}\bigg) +\bigg(  \frac{L_2 (1-\lambda_1)}{G_{12}(\lambda_2 -\lambda_1)}\bigg)  a_2 =: d_{21} + d_{22} a_2,\\
\end{align*}

Plugging into \eqref{eq-smoothfit-stopping}(v-vi), we get a two-equation system in the unknown $a_1,a_2$:
\begin{align*}
\begin{cases}
& (a_2/a_1)^{\gamma_1^+} (c_{11}+ c_{12} a_1)+ (a_2/a_1)^{\gamma_1^-} (c_{21} + c_{22} a_1) \\
& =(d_{11}+d_{21} + R_1-D) + (d_{12} + d_{22} +L_1-C)a_2 =: e_{11}+e_{12} a_2, \\[0.2cm]
&  (a_2/a_1)^{\gamma_1^+} (c_{11}+ c_{12} a_1) \gamma_1^+ + (a_2/a_1)^{\gamma_1^-} (c_{21} + c_{22} a_1) \gamma_1^- \\
&  =(d_{11}\lambda_1+d_{21}\lambda_2) + (d_{12}\lambda_1 + d_{22}\lambda_2 +L_1-C)a_2 =: e_{21} + e_{22}a_2.
\end{cases}  
\end{align*}

Solving with respect to $(a_2/a_1)^{\gamma_1^+}(c_{11}+ c_{12} a_1)$ and $(a_2/a_1)^{\gamma_1^-} (c_{21}+ c_{22} a_1)$, we get
\begin{equation*}
\begin{cases}
(a_2/a_1)^{\gamma_1^+} (c_{11}+ c_{12} a_1) = \frac{e_{21}- \gamma_1^- e_{11}}{\gamma_1^+ - \gamma_1^-} + \frac{e_{22} -  \gamma_1^- e_{12}}{\gamma_1^+ - \gamma_1^-} a_2 =: f_{11} + f_{12}a_2, \\[0.3cm]
(a_2/a_1)^{\gamma_1^-}(c_{21}+ c_{22} a_1) = \frac{e_{11}\gamma_1^+ - e_{21}}{\gamma_1^+ - \gamma_1^-} + \frac{e_{12} \gamma_1^+ - e_{22}}{\gamma_1^+ - \gamma_1^-} a_2 =: f_{21} + f_{22}a_2,
\end{cases}  
\end{equation*}
and finally
\begin{equation}
\label{eq-coeff-a}
\begin{cases}
a_1^{-\gamma_1^+} (c_{11}+ c_{12} a_1) = a_2^{-\gamma_1^+}(f_{11} + f_{12}a_2), \\
a_1^{-\gamma_1^-}(c_{21}+ c_{22} a_1) = a_2^{-\gamma_1^-}(f_{21} + f_{22}a_2).
\end{cases}  
\end{equation}
We resume the results in the following proposition.

\begin{prop}
\label{prop-explicit-v}
The following results hold.
\begin{itemize}
\item[(i)] The optimal threshold $a=(a_1,a_2) \in \rr^2_+$ of the optimal stopping problem \eqref{eq-pb-OS} is the unique non-negative solution to system \eqref{eq-coeff-a};
\item[(ii)] The value function $v$ of the optimal stopping problem \eqref{eq-pb-OS} is given by 
\begin{equation}
\label{eq-closed-v} 
\begin{aligned}
v(x,1)&=
\begin{cases}
\kappa, & x \leq a_1,\\[0.1cm]
A \big(\frac{x}{a_1}\big)^{\gamma_1^+} + B \big(\frac{x}{a_1}\big)^{\gamma_1^-} + C_1 x + D_1,  & x \in (a_1,a_2), \\[0.1cm]
M_1 \big(\frac{x}{a_2}\big)^{\lambda_1} + M_2 \big(\frac{x}{a_2}\big)^{\lambda_2} + L_1 x + R_1,  & x \geq a_2, 
\end{cases}
\\[0.15cm]
v(x,2)&=
\begin{cases}
\kappa, & x \leq a_2,\\[0.1cm]
- M_1 G_{11} \big(\frac{x}{a_2}\big)^{\lambda_1} - M_2  G_{12} \big(\frac{x}{a_2}\big)^{\lambda_2} + L_2 x + R_2,  & x > a_2, 
\end{cases}
\end{aligned}
\end{equation}
with $a_i$ as in (i) and with $\gamma_i, \lambda_i, A, B, C, D, M_i, L_i, R_i, G_{1i}$ defined above.
\end{itemize}
\end{prop}
\begin{proof}
The structure of $v$ provided in (ii), as well as the fact that the free boundaries solve system \eqref{eq-coeff-a} (cf.\ claim (i) of the proposition), follow by construction. Therefore, we just discuss the uniqueness claim. If there would be two distinct solutions to \eqref{eq-coeff-a}, then one could construct two distinct candidate value functions that, via verification, would identify to two distinct value functions of the optimal stopping problem. However, the value function is unique by definition.
\end{proof}

By Proposition \ref{prop-conclus-step1} (see in particular \eqref{eq-link-Vv}),  the expression of the value $V$ of the singular control problem is then easily found by suitably integrating the value $v$ of the optimal stopping problem. Thus, by \eqref{eq-link-Vv} and recalling $k_1$ and $k_2$ as in Proposition \ref{prop-conclus-step1}, one immediately obtains the following result.

\begin{corollary}
\label{cor:V}
Suppose that $1+\gamma_1^-\neq 0$ and $1+\lambda_i\neq 0$, for $i=1,2$. Then
\begin{align*}
V(x,1)&=
\begin{cases}
k_1 + \kappa (x - a_1), & x \leq a_1,\\
& \\
k_1 + \frac{A a_1}{1+\gamma_1^+} \Big[ \big(\frac{x}{a_1}\big)^{1+\gamma_1^+}  - 1 \Big]+ \frac{B a_1}{1+\gamma_1^-} \Big[\big(\frac{x}{a_1}\big)^{1+\gamma_1^-} - 1 \Big]   + \frac12 C_1 (x^2 - a_1^2 ) + D_1(x-a_1),  & x \in (a_1,a_2), \\ 
& \\
k_1 + \frac{A a_1}{1+\gamma_1^+} \Big[ \big(\frac{a_2}{a_1}\big)^{1+\gamma_1^+}  - 1 \Big]+ \frac{B a_1}{1+\gamma_1^-} \Big[\big(\frac{a_2}{a_1}\big)^{1+\gamma_1^-} - 1 \Big]  
 + \frac12 C_1 (a_2^2 - a_1^2 ) + D_1(a_2-a_1) \\ \\
 + \frac{M_1 a_2}{1+\lambda_1} \Big[ \big(\frac{x}{a_2}\big)^{1+\lambda_1} - 1\Big]
+ \frac{M_2 a_2}{1+\lambda_2} \Big[ \big(\frac{x}{a_2}\big)^{1+\lambda_2} - 1  \Big] \\ \\
+ \frac12  L_1 (x^2 - a_2^2) + R_1(x-a_2),  & x \geq a_2, 
\end{cases}\\
& \\
V(x,2)&=
\begin{cases}
k_2 + \kappa(x-a_2) , & x \leq a_2,\\
& \\
k_2 - \frac{M_1 G_{11} a_2}{1+\lambda_1} \Big[  \big(\frac{x}{a_2}\big)^{1+\lambda_1}  -1 \Big] - \frac{M_2  G_{12} a_2}{1+\lambda_2} \Big[ \big(\frac{x}{a_2}\big)^{1+\lambda_2} -1 \Big] \\ \\ +\frac12  L_2 (x^2 -a_2^2) + R_2(x-a_2),  & x > a_2, 
\end{cases}
\end{align*}

On the other hand, if $1 + \zeta =0$ for $\zeta \in \{\gamma_1^-, \lambda_1, \lambda_2\}$, then the terms of the form $\frac{1}{1+\zeta}[(\cdot)^{1+\zeta} - 1]$ in the formulas above should be replaced by $\ln(\cdot)$.
\end{corollary}


\subsection{Second step: Determining the stationary distribution}
\label{ssec-step2}

In Section \ref{ssec-step1} we computed, for $Q=(Q_1,Q_2) \in \rr^2_+$ given and fixed, the optimal control $I^{\star}(Q)$. We now determine the stationary distribution $(p_\infty^Q(dx,i))_{i = 1,2}$ of the joint process $(X^{I^{\star}(Q)}_t, \eps_t)_{t \geq 0}$. Once again, we fix $Q=(Q_1,Q_2) \in \rr^2_+$ and omit the dependence on $Q$ to ease the notation burden. 

To simplify the computations, instead of working with $X^{I^{\star}}$, we will work with its natural logarithm
\begin{equation*}
Z_t:=\ln(X^{I^{\star}}_t).  
\end{equation*}
Notice that by \eqref{eq-SDE-sol}, for $x>0$, we have
\begin{equation*}
dZ_t = - \Big(\delta + \frac12 \sigma^2_{\eps_t}\Big)dt + \sigma_{\eps_t} dB_t +  dI^{\star}_t, \qquad t \geq 0, \qquad Z_{0-}=\ln(x),
\end{equation*}
where $I^{\star}$ is such that $Z_t \geq b_{\eps_t}$ and that $I^{\star}_t=\int_0^t \mathbbm{1}_{\{Z_s \leq b_{\eps_s}\}} dI^{\star}_t$, with
\begin{equation}
\label{eq:abar}
b_i := \ln(a_i), \quad i \in \{1,2\}.
\end{equation}
Notice that, under Assumption \eqref{eq-order-a}, we have $b_1 < b_2$. By \cite[Theorem 2]{BudLiu}, the pair $(Z_t, \eps_t)_{t \geq 0}$ is a positively recurrent process and thus admits a unique stationary distribution $\Pi$. By a slight abuse of notation we write  
\begin{equation*}
\Pi(z,i) = \mathbb{P}(Z_\infty \leq z, \eps_\infty=i), \qquad i \in \{1,2\}.
\end{equation*}
Furthermore, by adapting \cite[Theorem 1]{DAuriaKella} to our setting with only one regime-dependent reflecting boundary, and recalling $P_{ij}$ and $\pi_i$ as in \eqref{eq-Pi}, we have the following characterization: For $i \in \{1,2\}$, $\Pi$ is the unique solution with non-decreasing components to 
\begin{equation}
\label{eq-pde-stationary}
\begin{cases}
\Pi(z,i)=0, & z < b_i, \\
\frac12 \sigma_i^2 \Pi''(z,i) + \Big(\delta + \frac12 \sigma^2_i\Big) \Pi'(z,i) + \sum_{j=1}^2 P_{ij} \Pi(z,j) = 0, & z \geq b_i, \\
\Pi(z,i)= \pi_i, & \text{$z \to +\infty$,}
\end{cases}
\end{equation}
such that $\Pi(\cdot,1) \in C^0(\rr) \cap C^1(\rr \setminus \{b_1\}) \cap C^2(\rr \setminus \{b_1, b_2\})$ and $\Pi(\cdot,2) \in C^0(\rr) \cap C^1(\rr \setminus \{b_2\}) \cap C^2(\rr \setminus \{b_2\})$.

In order to determine an explicit expression for $\Pi$, we will first build a parametric class of solutions to \eqref{eq-pde-stationary}, then set the free parameters by imposing the regularity and asymptotic conditions, and finally check that the candidate function is indeed non-decreasing. 

We solve \eqref{eq-pde-stationary} by considering separately the three intervals $(-\infty,b_1)$, $(b_1, b_2)$, $(b_2,+\infty)$. The case $z \in (-\infty,b_1)$ is immediate: 
\begin{equation*}
\Pi(z,1) = 0 = \Pi(z,2).    
\end{equation*}
When $z \in (b_1, b_2)$, \eqref{eq-pde-stationary} implies that 
\begin{equation*}
\begin{cases}
\frac12 \sigma_1^2 \Pi''(z,1) + \Big(\delta + \frac12 \sigma^2_1\Big) \Pi'(z,1) - p_1 \Pi(z,1) = 0, \\
\Pi(z,2) = 0.
\end{cases}    
\end{equation*}
The first equation is an homogeneous second-order linear ODE, whose solution is 
\begin{equation*}
\Pi(z,1) = A_1 e^{\alpha_1^+ (z - b_1)} + A_2 e^{\alpha_1^- (z - b_1)},
\end{equation*}
where $A_1,A_2 \in \rr$ are free parameters and where $\alpha_1^-,\alpha_1^+$ are the two solution to $\phi_1(\alpha)=0$, with 
\begin{equation}
\label{eq:phii}
\phi_i(\alpha) = \frac12 \sigma_i^2\alpha^2 + \Big(\delta + \frac12 \sigma^2_i\Big)\alpha - p_i, \qquad i \in \{1,2\};
\end{equation}
that is,
\begin{equation}
\label{eq:alpha1pm}
\alpha_{1}^\pm= \frac{-(\delta + \frac12 \sigma^2_1) \pm \sqrt{(\delta + \frac12 \sigma^2_1)^2 + 2 \sigma_1^2p_1}}{\sigma_1^2},\qquad\text{with}\quad \alpha_1^-<0<\alpha_1^+. 
\end{equation}
Finally, for $z \in (b_2, +\infty)$, \eqref{eq-pde-stationary} implies that
\begin{equation*}
\begin{cases}
\frac12 \sigma_1^2 \Pi''(z,1) + \Big(\delta + \frac12 \sigma^2_1\Big) \Pi'(z,1) - p_1 \Pi(z,1) + p_1 \Pi(z,2) = 0, \quad \lim_{z \uparrow \infty}\Pi(z,1) = \pi_1, \\[0.3cm]
\frac12 \sigma_2^2 \Pi''(z,2) + \Big(\delta + \frac12 \sigma^2_2\Big) \Pi'(z,2) - p_2 \Pi(z,2) + p_2 \Pi(z,1) = 0, \quad \lim_{z \uparrow \infty}\Pi(z,2) = \pi_2, 
\end{cases}    
\end{equation*} 
whose general solution is given by
\begin{equation*}
\begin{cases}
\Pi(z,1) = \sum_{i=1}^4 B_i e^{\theta_i (z - b_2)} + \Pi_{\text{part}}(z,1),\\
\Pi(z,2) = \sum_{i=1}^4 \bar B_i e^{\theta_i (z - b_2)} + \Pi_{\text{part}}(z,2),
\end{cases}    
\end{equation*}
where $B_i \in \rr$ are free parameters, $\Pi_{\text{part}}$ are particular solutions to the system of ODEs, $\theta_1 < \theta_2 < \theta_3 = 0 < \theta_4$ are the four roots of 
\begin{align}\label{eq:thetas}
\bar \phi (\theta) := \phi_1(\theta)\phi_2(\theta) - p_1p_2 = 0,
\end{align}
and (cf.\ \eqref{eq:phii})
\begin{equation}
\label{eq:Bbar}
\bar B_i = - \frac{\phi_1(\theta_i)}{p_2} B_i =: - \phi_{1i} B_i, \qquad\quad i \in \{1,2,3,4\}.
\end{equation}
Notice that the existence of the $\theta_i$s above is guaranteed by the following properties: (i) $\bar \phi(\alpha_i) = -p_1p_2 <0$, with $\alpha_2<0<\alpha_1$; (ii) $\bar \vphi(0)=0$ and $\bar \phi'(0)<0$; (iii) $\bar \phi(+\infty) = \bar \phi(-\infty) = +\infty$. 

As for $\Pi_{\text{part}}$, by looking for constant solutions, we notice that 
\begin{equation*}
\Pi_{\text{part}}(z,i) = \pi_i, \qquad (z,i) \in \rr \times \{1,2\}
\end{equation*}
satisfies both the equation and the limit condition. Furthermore, in order to ensure that $\lim_{z \uparrow \infty}\Pi(z,i)=\pi_i$ is satisfied, we set $B_3=B_4=0$, so that we are left with
\begin{equation*}
\begin{cases}
\Pi(z,1) = B_1 e^{\theta_1 (z - b_2)} + B_2 e^{\theta_2 (z - b_2)} + \pi_1,\\[0.1cm]
\Pi(z,2) = - B_1 \phi_{11} e^{\theta_1 (z - b_2)} - B_2  \phi_{12} e^{\theta_2 (z - b_2)} + \pi_2.
\end{cases}    
\end{equation*}
To determine $A_1,A_2,B_1,B_2$, we finally impose that $\Pi(\cdot,1) \in C^0(\rr) \cap C^1(\rr \setminus \{b_1\}) \cap C^2(\rr \setminus \{b_1, b_2\})$ and $\Pi(\cdot,2) \in C^0(\rr) \cap C^1(\rr \setminus \{b_2\}) \cap C^2(\rr \setminus \{b_2\})$; that is,
\begin{equation*}
\begin{cases}
\Pi(b_{1}-,1) = \Pi(b_{1}+,1), \\
\Pi(b_{2}-,1) = \Pi(b_{2}+,1), \\
\Pi'(b_{2}-,1) = \Pi'(b_{2}+,1), \\
\Pi(b_{2}-,2) = \Pi(b_{2}+,2), 
\end{cases}
\end{equation*}
which is, in turn, equivalent to the linear system
\begin{equation}
\label{eq-stationary-system}
\begin{bmatrix}
1 & 1 & 0 & 0 \\[0.1cm]
e^{\alpha_1^+ (b_2 - b_1)} & e^{\alpha_1^- (b_2 - b_1)} & -1 & -1 \\
\alpha_1^+ e^{\alpha_1^+ (b_2 - b_1)} & \alpha_1^- e^{\alpha_1^- (b_2 - b_1)} & -\theta_1 & -\theta_2 \\[0.1cm] 
0 & 0 & \phi_{11} &\phi_{12} 
\end{bmatrix}
\begin{bmatrix}
A_1\\
A_2\\
B_1\\
B_2
\end{bmatrix}
=
\begin{bmatrix}
0\\
\pi_1\\
0\\
\pi_2
\end{bmatrix}
.
\end{equation}
It is easy to check that System \eqref{eq-stationary-system} admits a unique solution.

The next proposition shows that $\Pi$ constructed above does indeed identify with the stationary distribution of the process $(X^{I^{\star}}_t, \eps_t)_{t \geq 0}$.

\begin{prop}
\label{prop-stationary}
Recall $p_i,\pi_i$ as in \eqref{eq-Pi}, $b_i$ as in \eqref{eq:abar}, $\alpha_1^{\pm}$ as in \eqref{eq:alpha1pm}, $\theta_i$ as the solutions to \eqref{eq:thetas}, $A_i, B_i$ uniquely solving \eqref{eq-stationary-system}, and $\phi_{1i}$ as in \eqref{eq:Bbar} (cf.\ also \eqref{eq:phii}).

The stationary cumulative distribution function of $(Z_t, \eps_t)_{t \geq 0}$, is given by
\begin{equation}
\label{eq-stationary-CDF}
\begin{aligned}
\Pi(z,1) &=
\begin{cases}
0, & z \leq b_1,\\
A_1 e^{\alpha_1^+ (z - b_1)} + A_2 e^{\alpha_1^- (z - b_1)},  & z \in (b_1, b_2), \\ 
B_1 e^{\theta_1 (z - b_2)} + B_2 e^{\theta_2 (z - b_2)} + \pi_1,  & z \geq b_2, \\
\end{cases}
\\[0.15cm]
\Pi(z,2) &=
\begin{cases}
0, & z \leq b_2,\\[0.1cm]
- B_1\phi_{11} e^{\theta_1 (z - b_2)} - B_2 \phi_{12} e^{\theta_2 (z - b_2)} + \pi_2,  & z > b_2.\\
\end{cases}
\end{aligned}
\end{equation}

Consequently, the stationary distribution of the joint process $(X^{I^{\star}}_t, \eps_t)_{t \geq 0}$ is given by 
\begin{equation}
\label{eq-stationary-measure}
\widehat{\Pi}(x,i) := \Pi(\ln(x),i) = \int_{-\infty}^{x} \frac{\Pi'(\ln(y),i)}{y} dy =: \int_{0}^{x} p_{\infty}(dy,i), \qquad \quad (x,i) \in \rr_+ \times \{1,2\}.
\end{equation}
\end{prop}
\begin{proof}
Postponed to Appendix \ref{sec-appendix}.
\end{proof}

Explicit calculations on \eqref{eq-stationary-CDF} and  \eqref{eq-stationary-measure} yield the following.
\begin{corollary}
\label{cor:exlicitPihat}
Recall \eqref{eq-stationary-CDF} and  \eqref{eq-stationary-measure}. One has:
\begin{equation}
\label{eq-stationary-CDF-X}
\begin{aligned}
\widehat{\Pi}(x,1) &=
\begin{cases}
0, & x \leq  a_1,\\
A_1 \Big( \frac{x}{a_1} \Big)^{\alpha_1^+} + A_2 \Big( \frac{x}{a_1} \Big)^{\alpha_1^-} ,  & x \in ( a_1,  a_2), \\ 
B_1  \Big( \frac{x}{a_2} \Big)^{\theta_1} + B_2 \Big( \frac{x}{a_2} \Big)^{\theta_2}+ \pi_1,  & x \geq  a_2, \\
\end{cases}
\\[0.15cm]
\widehat{\Pi}(x,2) &=
\begin{cases}
0, & x \leq  a_2,\\[0.1cm]
- B_1 \phi_{11} \Big( \frac{x}{a_2} \Big)^{\theta_1}  - B_2  \phi_{12} \Big( \frac{x}{a_2} \Big)^{\theta_2}  + \pi_2,  & x > a_2,\\
\end{cases}
\end{aligned}
\end{equation}

Furthermore, if $1+\theta_2<0$,
\begin{align}
\label{eq:average1}
\int_0^\infty x p_\infty(dx,1)  = \frac{A_1 \alpha_1^+ a_1}{1+\alpha_1^+} \Big[ \Big( \frac{a_2}{a_1}\Big)^{1+\alpha_1^+} -  1  \Big] +
\frac{A_2 \alpha_1^- a_1}{1+\alpha_1^-} \Big[ \Big( \frac{a_2}{a_1}\Big)^{1+\alpha_1^-} -  1 \Big] 
 - \frac{ B_1 \theta_1 a_2  }{1+\theta_1}  
   - \frac{ B_2 \theta_2 a_2}{1+\theta_2}
\end{align}
and
\begin{align}
\label{eq:average2}
\int_{0}^\infty x p_\infty(dx,2)  =  \Big[\frac{B_1 \theta_1 \phi_{11}}{1+\theta_1}   
+  \frac{B_2 \theta_2 \phi_{12}}{1+\theta_2}\Big] a_2,  
\end{align}
\end{corollary}

\subsection{Third step: Solution to the fixed-point problem}
\label{ssec-step3}

In Section \ref{ssec-step2} we computed, for given $Q=(Q_1,Q_2) \in \rr^2_+$, the stationary distribution $(p_\infty^Q(dx,i))_{i = 1,2}$ of the joint process $(X^{I^{\star}(Q)}_t, \eps_t)_{t \geq 0}$. We now look for solutions $Q^{\star}=(Q^{\star}_1,Q^{\star}_2) \in \rr^2_+$ to the fixed-point problem 
\begin{equation}
\label{eq-fixedpoint}
\begin{cases}
Q_1^{\star} = \displaystyle{ \frac{1}{\pi_1} \int_0^\infty x \,\, p^{Q^{\star}}_\infty(dx,1), } \\[0.5cm]
Q_2^{\star} = \displaystyle{ \frac{1}{\pi_2} \int_0^\infty x \,\, p^{Q^{\star}}_\infty(dx,2) }.
\end{cases}
\end{equation}


\begin{thm}
Let Assumption \ref{ass:assrho} hold and assume $1+\theta_2<0$. There exists a unique solution $Q^{\star}=(Q^{\star}_1,Q^{\star}_2) \in \rr^2_+$ to system \eqref{eq-fixedpoint}. As a result, $(I^{\star}(Q^{\star}), Q^{\star})$ is the unique stationary mean-field equilibrium in the sense of Definition \ref{def-MFE}.
\end{thm}

\begin{proof} 

We first prove existence and then uniqueness.
\vspace{0.25cm}

\textit{1. Existence}. Define the map $\mathcal{R}:\rr^2_+\to\rr^2_+$ as
\begin{equation*}
\mathcal{R}Q = ((\mathcal{R}Q)_1, (\mathcal{R}Q)_2) := \bigg(\frac{1}{\pi_1}\int_0^\infty x \, p^Q_\infty(dx,1), \,\, \frac{1}{\pi_2}\int_0^\infty x \, p^Q_\infty(dx,2) \bigg),  
\end{equation*}
for each $Q=(Q_1,Q_2) \in \rr^2_+$. Then, we are looking for a fixed point $Q^{\star}$ for $\mathcal{R}$; that is, $Q^{\star}$ such that
\begin{equation*}
\mathcal{R}Q^{\star} = Q^{\star}, \qquad\text{i.e.,}\qquad 
\begin{cases}
Q_1^{\star} = (\mathcal{R}Q^{\star})_1, \\
Q_2^{\star} = (\mathcal{R}Q^{\star})_2.
\end{cases}
\end{equation*}

We will prove the existence of such a $Q^{\star}$ by means of the Brower's fixed point theorem and the rest of this proof is organized in several steps. In the following, we will fully specify the dependence on $Q=(Q_1,Q_2)$ of the involved quantities, whenever necessary.
\vspace{0.15cm}

\textbf{Step 1: Monotonicity and continuity of the free boundaries w.r.t.\ $Q$.}
Recalling \eqref{eq-fbp-OS}, it is easy to see that, for any $i=1,2$, $Q_1 \mapsto v^Q(x,i)$ and $Q_2 \mapsto v^Q(x,i)$ are decreasing. Since
\begin{equation*}
a_i(Q) = \inf \Big\{ x>0: v^Q(x,i) < \kappa\Big\},
\end{equation*}
we then have that $Q_1 \mapsto a_i(Q)$ and $Q_2 \mapsto a_i(Q)$ are decreasing as well for any $i=1,2$.

Furthermore, recalling that the boundaries $a_i(Q)$ uniquely solve System \eqref{eq-coeff-a} -- which involves continuously differentiable functions -- the implicit function theorem yields continuity of the functions $Q \mapsto a_i(Q)$, $i \in \{1,2\}$; that is, if $Q^n \to Q$, then $a_1(Q^n) \to a_1(Q)$ and $a_2(Q^n) \to a_2(Q)$. 
\vspace{0.15cm}

\textbf{Step 2: Continuity of $Q \mapsto \mathcal{R}Q$.}
Given that $p^Q_{\infty}$ is absolutely continuous with respect to the Lebesgue measure, with density $s^Q_{\infty}(x,i):=\frac{1}{x}(\Pi^Q)'(\ln(x),i)$ (cf.\ \eqref{eq-stationary-measure} and \eqref{eq-stationary-CDF}), continuity of $a_i \mapsto s^Q_{\infty}(x,i)$ can be shown by direct check for any $(x,i) \in \rr_+ \times \{1,2\}$. Thanks to Step 1, this in turn yields the continuity of $Q \mapsto s_{\infty}^Q(x,i)$ for any $(x,i) \in \rr_+ \times \{1,2\}$. By dominated convergence theorem, this then ensures that $Q\mapsto\mathcal{R}Q$ is therefore continuous. 
\vspace{0.15cm}

\textbf{Step 3: Uniform bounds for $\mathcal{R}Q$.} If a fixed point exists, then this should necessarily live in the image of $Q \mapsto \mathcal{R}Q$. Here we determine uniform bounds for $\mathcal{R}Q$.

From the optimal stopping problem \eqref{eq-pb-OS}, an integration by parts gives
\begin{equation*}
v^Q(x,i) = \kappa + \inf_{\tau \geq 0} \bar{\eee}_{(x,i)}\bigg[ \int_0^\tau e^{-(\rho+\delta)s}\Big( \eta^Q_{\eps_s} - 2c\bar X_s -  (\rho+\delta) \Big) ds \bigg],
\end{equation*}
with $\bar X$ as in \eqref{eq-Xbar}. Since $\zeta_i,Q_i \geq 0$, we have that $\eta_i^Q = \vphi_i + \zeta_i Q_i^{-\alpha} \geq \vphi_i$. Also, recall that $\vphi_i > (\rho+\delta)$ by assumption, for $i \in \{1,2\}$, so that
\begin{equation*}
v^Q(x,i) \geq \kappa + \inf_{\tau \geq 0} \bar{\eee}_{(x,i)}\bigg[ \int_0^\tau e^{-(\rho+\delta)s}\Big( \vphi_i - (\rho+\delta) - 2c\bar X_s \Big) ds \bigg] =: \underline v(x,i). 
\end{equation*}
In particular, this implies that, for each $Q \in \rr^2_+$,
\begin{equation*}
a_i(Q) = \inf \Big\{ x>0: v(x,i) < \kappa\Big\} 
\geq \inf \Big\{ x>0: \underline v(x,i) < \kappa\Big\} =: \underline a_i \in \Big(0, \frac{1}{2c}\big(\vphi_i - (\rho+\delta)\big)\Big).
\end{equation*}
The inclusion on the very right-hand side of the equation above follows from the fact that, for problem $\underline v$, it is never optimal to stop on the region $\{(x,i) \in \rr_+ \times \{1,2\}:\, 2cx > \varphi_i - (\rho+\delta) \}$. Note that $\underline a = (\underline a_1,\underline a_2)$ does not depend on $Q$, as in fact $\underline v$ does not depend on $Q$.

With an abuse of notation, we now highlight the dependence of reflected processes with respect to the reflection barriers. In particular, we denote by $X^{a(Q)}$ the process $X^{I^{\star}(Q)}$ (which is reflected upwards at the barrier $a_{\eps}(Q)$), and by $X^{\underline a}$ the process that is reflected upwards at the barrier $\underline{a}_{\eps}$; that is, $(X^{\underline a}, \eps)$ satisfies \eqref{eq-Skrefl} with $a_{\eps}(Q)$ replaced by $\underline{a}_{\eps}$, and, in particular, $X^{\underline a}$ is given through \eqref{eq-optctrl-SC} with $a_{\eps}(Q)$ replaced by $\underline{a}_{\eps}$. By comparison (through a penalization argument) for reflected SDEs, one has 
\begin{equation*}
X^{a(Q)}_t \geq X^{\underline a}_t, \quad \mathbb{P}-\text{a.s.}\,\, \forall t\geq0,
\end{equation*}
which yields, for $i \in \{1,2\}$,
\begin{equation*}
\eee \big[X_\infty^{a(Q)} \big| \eps_\infty = i \big] \geq \eee \big[X_\infty^{\underline a} \big| \eps_\infty = i \big];
\end{equation*}
that is,
\begin{equation*}
\frac{1}{\pi_i}\int_0^\infty x \, p^Q_\infty(dx,i) \geq \frac{1}{\pi_i}\int_0^\infty x \, \underline p_\infty(dx,i) =: \underline Q_i,
\end{equation*}
where $\underline p_\infty$ is the stationary distribution of $(X^{\underline a}_t, \eps_t)_{t \geq 0}$ (which can be explicitly constructed as in Section \ref{ssec-step2}). Notice that $\underline Q = (\underline Q_1, \underline Q_2)$ does not depend on $Q$, as $\underline a$ does not either. Hence, without loss of generality, we can restrict to those $Q=(Q_1,Q_2)$ such that $Q_i \geq \underline Q_i$ for $i \in \{1,2\}$. 

On the other hand, by Step 1 we have that $Q_1 \mapsto a_i(Q)$ and $Q_2 \mapsto a_i(Q)$ are decreasing, for $i \in \{1,2\}$. Hence, 
\begin{equation*}
a_i(Q) \leq a_i(\underline Q) =: \overline a_i,
\end{equation*}
for each $Q \in \rr^+_2$. Notice that, again, $\overline a_i$ does not depend on $Q$. Therefore, denoting by by $X^{\overline a}$ the process that is reflected upwards at the barrier $\overline{a}_{\eps}$, arguing as before,
\begin{equation*}
\eee \big[X_\infty^{a(Q)} \big| \eps_\infty = i \big] \leq \eee \big[X_\infty^{\overline a} \big| \eps_\infty = i \big];
\end{equation*}
that is,
\begin{equation*}
\frac{1}{\pi_i}\int_0^\infty x \, p^Q_\infty(dx,i) \leq \frac{1}{\pi_i}\int_0^\infty x \, \overline p_\infty(dx,i) =: \overline Q_i,
\end{equation*}
where $\overline p_\infty$ is the stationary distribution of $(X^{\overline a}_t, \eps_t)_{t \geq 0}$ (which, again, can be constructed as in Section \ref{ssec-step2}). It thus follows that, without loss of generality, we can restrict to those $Q=(Q_1,Q_2)$ such that $Q_i \leq \overline Q_i$ for $i \in \{1,2\}$. 

Overall, if an equilibrium exists, then this should necessarily live within the interval $[\underline Q_1, \overline Q_2] \times [\underline Q_1, \overline Q_2]$. Without loss of generality we can thus restrict our attention only to those $Q=(Q_1,Q_2)$ such that
\begin{equation*}
(Q_1,Q_2) \in [\underline Q_1, \overline Q_2] \times [\underline Q_1, \overline Q_2].   
\end{equation*}
\vspace{0.15cm}

\textbf{Step 4: Applying Brower's fixed point theorem.} Finally, combining Step 2 and Step 3, Brower's fixed point theorem guarantees that there exists a fixed point $Q^{\star}=(Q^{\star}_1,Q^{\star}_2)$ for map $\mathcal{R}$.
\vspace{0.25cm}

\textit{2. Uniqueness}. Suppose that there exist $Q^{\star}=(Q^{\star}_1,Q^{\star}_2)$ and $\widetilde{Q}=(\widetilde{Q}_1,\widetilde{Q}_2)$ solving fixed point problem \eqref{eq-fixedpoint} and such that, for example, $\widetilde{Q}_1 > Q^{\star}_1$, while $\widetilde{Q}_2 = Q^{\star}_2$. Then, by Step 1 in the proof of the existence claim, $a_i(\widetilde{Q}) < a_i(Q^{\star})$, for $i\in\{1,2\}$, which gives (arguing as in Step 3 of the proof of the existence claim)
\begin{equation*}
\eee \big[X_\infty^{a(\widetilde{Q})} \big| \eps_\infty = 1 \big] \leq \eee \big[X_\infty^{a({Q}^{\star})} \big| \eps_\infty = 1 \big].
\end{equation*}
Hence, $\widetilde{Q}_1 \leq Q^{\star}_1$, thus a contradiction.
\end{proof}

%
%


\medskip


\indent \textbf{Acknowledgments.} The first and the second authors thank the Chairs EDF--CA CIB Finance and Sustainable Development: Quantitative Approach and the Finance for Energy Markets Research Initiative for their support. Furthermore, they acknowledge the support by the French ANR EcoREES under grant ANR-19-CE05-0042. The third author gratefully acknowledges financial support by the \emph{Deutsche Forschungsgemeinschaft} (DFG, German Research Foundation) - SFB 1283/2 2021 - 317210226.


\appendix
\section{Some proofs}
\label{sec-appendix}

\begin{proof}[Proof of Proposition \ref{prop-verif-SC}]
The proof is organized in two steps. We first show that $I^{\star}$ is admissible and then we prove the optimality of the candidate solution $w$ and $I^{\star}$.
\vspace{0.25cm}

\emph{Admissibility of $I^{\star}$.} Clearly, $I^{\star}$ as in \eqref{eq-optctrl-SC} is $\mathbb{F}$-adapted, it has right-continuous and nondecreasing paths, and it is such that $I^{\star}_{0-}=0$. It thus remains to show that it satisfies the limit and the integrability conditions in \eqref{set:A}. In the sequel, let $X^{\star}:=X^{I^\star}$. 

We start by proving that $\lim_{T\uparrow \infty}\eee_{(x,i)}\big[e^{-\rho T} |X^{\star}_T|^2\big]=0$. Let $\overline{\sigma}:=\max\{\sigma_1,\sigma_2\}$. By using \eqref{eq-SDE-sol}, it is easy to see that 
\begin{equation}
\label{eq:estimlim-1}
\eee_{(x,i)}\big[|X^{\star}_T|^2\big] \leq x^2 e^{-2\delta T + \overline{\sigma}^2 T}\widehat{\eee}_{(x,i)}\Big[e^{2I^{\star}_T}\Big],
\end{equation}
where $\widehat{\eee}_{(x,i)}$ denotes the expectation, conditional on $(X^{\star}_{0-}, \varepsilon_{0-})=(x,i) \in \rr_+ \times \{1,2\}$, with respect to the probability measure $\widehat{\mathbb{P}}$, defined on $(\Omega, \mathcal{F}_T)$ and such that 
$$\frac{d\widehat{\mathbb{P}}}{d\mathbb{P}} = e^{-2\int_0^T\sigma^2_{\varepsilon_u} du + 2\int_0^T\sigma_{\varepsilon_u} dB_u}.$$
Also, let $\widehat{B}_t:=B_t - \int_0^t\sigma_{\varepsilon_u} dB_u$, which, by Girsanov theorem, is an $\mathbb{F}$-Brownian motion under $\widehat{\mathbb{P}}$.

Then, by employing \eqref{eq-optctrl-SC} and letting $\overline{a}:=\max\{a_1,a_2\}$, we find that, under $\widehat{\mathbb{P}}_{(x,i)}$,
\begin{equation}
\label{eq:estimlim-2}
I_T \leq \ln(\frac{\overline{a}}{x}) \vee 0 + \delta T + \sup_{0 \leq s \leq T}\Big(- \frac{1}{2}\int_0^s \sigma^2_{\varepsilon_u} du - \int_0^s\sigma_{\varepsilon_u} d\widehat{B}_u\Big).
\end{equation}
By plugging \eqref{eq:estimlim-2} into \eqref{eq:estimlim-1}, we obtain
\begin{equation}
\label{eq:estimlim-3}
\eee_{(x,i)}\big[|X^{\star}_T|^2\big] \leq C(x) e^{\overline{\sigma}^2 T}\widehat{\eee}_{(x,i)}\Big[\sup_{0 \leq s \leq T} \Big( e^{- \frac{1}{2}\int_0^s \sigma^2_{\varepsilon_u} du - \int_0^s\sigma_{\varepsilon_u} d\widehat{B}_u}\Big)^2\Big],
\end{equation}
with $C(x):= x^2 e^{2 \ln(\frac{\overline{a}}{x}) \vee 0 }.$
Then, Burkholder-Davis-Gundy's inequality applied to the martingale $(e^{- \frac{1}{2}\int_0^s \sigma^2_{\varepsilon_u} du - \int_0^s\sigma_{\varepsilon_u} d\widehat{B}_u})_{s\geq 0}$ yields, after simple estimates,
\begin{equation}
\label{eq:estimlim-4}
\eee_{(x,i)}\big[|X^{\star}_T|^2\big] \leq \widehat{C}(x) \overline{\sigma}^2 T e^{2\overline{\sigma}^2 T},
\end{equation}
with $\widehat{C}(x):=C_o C(x)$, for some $C_o>0$.
It thus follows that $\lim_{T\uparrow \infty}\eee_{(x,i)}\big[e^{-\rho T} |X^{\star}_T|^2\big]=0$ if $\rho > 2\overline{\sigma}^2$, as required in Assumption \ref{ass:assrho}.
\vspace{0.1cm}

We now show that $\eee_{(x,i)}\big[\int_0^{\infty} e^{-\rho t}  \big(\big(X^{\star}_t)^2 dt + X^{\star}_t \circ dI^\star_t\big)\big]<\infty$. This is accomplished by adapting arguments from \cite{Shreveetal} (see also Lemma 5.2 in \cite{FeRo}).

Let $g:\mathbb{R}\times \{1,2\} \to \mathbb{R}$ be the solution to
$$\big(\mathcal{L}-\rho\big)g(x,i) + x^2=0,$$
such that, for any $(x,i)\in \rr \times \{1,2\}$, $|g(x,i)| \leq C(1 + |x|^2)$ for some $C>0$, and $g_x(a^{\star}_i,i)=-x$, for any $x\leq a^{\star}_i$, $i\in\{1,2\}$.

Then, taking a fixed $T>0$, by the regularity of $g$ we can apply It\^o-Meyer's formula for semimartingales to the Markov-modulated process $(e^{-\rho t} g(X^{\star}_{t}, \eps_{t}))_{t\geq0}$ on the time interval $[0,T]$ (see \cite{Eisenberg}) and obtain
\begin{align}
\label{adm2}
&\eee_{(x,i)}\Big[e^{-\rho T} g(X^{\star}_T,\eps_T)\Big] - g(x,i) = \eee_{(x,i)}\bigg[ - \int_0^T e^{-\rho t} \big(X^{\star}_t\big)^2 dt + \int_0^T e^{-\rho t} X^{\star}_t g_x(X^{\star}_t,\eps_t) dI^{\star,c}_t\bigg] \nonumber \\
&\quad + \eee_{(x,i)}\bigg[\sum_{0\le t \le T} e^{-\rho t}\left(g(X^{\star}_{t}, \eps_{t})-g(X^{\star}_{t-},\eps_{t})\right)\bigg].
\end{align}
Here, and in the following, $I^{\star,c}$ denotes the continuous part of $I^{\star}$.
Observe that, the latter expectation in \eqref{adm2} can be written as
\begin{align}
\label{sum-jumps}
&\eee_{(x,i)}\bigg[\sum_{0\leq t \leq T} e^{-\rho t} \big(g(X^{\star}_t, \eps_t)-g(X^{\star}_{t-},\eps_t) \big)\bigg] = \eee_{(x,i)}\bigg[\sum_{0\leq t \leq T} e^{-\rho t} \mathds{1}_{\{\Delta I^\star_{t}>0\}}\big(g(e^{\Delta I^\star_{t}}X^{\star}_{t-}, \eps_t)-g(X^{\star}_{t-},\eps_{t}) \big)\bigg] \nonumber\\
& = \eee_{(x,i)}\bigg[\sum_{0\leq t \leq T} e^{-\rho t} X^{\star}_{t-} \int_0^{{\Delta I^\star_{t}}} e^u g_x(e^u X^{\star}_{t-}, \eps_{t}) du
\bigg]. 
\end{align}

Impose now that $g_x(a^{\star}_i,i)=-1$, and extend the function $g$ on $(-\infty,a^{\star}_i)$ so that $g_x(x,i)=-1$ for any $x < a^{\star}_i$ (for example, set $g(x,i):=a^{\star}_i-x + g(a^{\star}_i,i)$ for $x < a(i)$). Then, since $I^\star_{\cdot}$ is flat off $\{t\geq 0: X^{\star}_{t-} \leq a^{\star}_{\eps_t}\}$, we get 
\begin{align}
\label{xietacont}
X^{\star}_{t-} \int_0^{\Delta I^{\star}_{t}} e^u g_x(e^u X^{\star}_{t-}, \eps_{t}) du = - X^{\star}_{t-} \int_0^{\Delta I^{\star}_{t}} e^u du.
\end{align}
Therefore, by substituting \eqref{xietacont} in \eqref{sum-jumps} and then in \eqref{adm2}, we get that (cf.\ \eqref{eq-XcircI}) 
\begin{equation}
\label{adm2-bis}
g(x,i) = \eee_{(x,i)}\Big[e^{-\rho T} g(X^{\star}_T,\eps_T)\Big] + \eee_{(x,i)}\bigg[\int_0^T e^{-\rho t}  \Big(\big(X^{\star}_t)^2 dt + X^{\star}_t \circ dI^\star_t\Big)\bigg].
\end{equation}

Finally, given that $|g(X^{\star}_T,\eps_T)| \leq C(1+|X^{\star}_T|^2)$, $\mathbb{P}_{(x,i)}$-a.s., we can let $T\uparrow \infty$ on the right-hand side of \eqref{adm2-bis}, apply the result of Step 1 on the fist expectation, the monotone convergence theorem on the second expectation, and obtain 
$$g(x,i) = \eee_{(x,i)}\bigg[\int_0^{\infty} e^{-\rho t}  \Big(\big(X^{\star}_t)^2 dt + X^{\star}_t \circ dI^\star_t\Big)\bigg].$$
The finiteness of the function $g$ constructed above, yields the claim.
\vspace{0.25cm}

\emph{The verification theorem.} The proof of the verification theorem is nowadays standard: See, e.g., \cite{Cadenillas}, \cite{FeRo} and \cite{FeSchZhu} for problems of singular stochastic control with regime switching, as well as \cite{GuoZervos} for a problem with geometric state-dynamics similar to ours, but without regime-switching. We provide it below for the sake of completeness.

Let $I\in \mathcal{A}$ and $w:\mathbb{R}\times \{1,2\} \to \mathbb{R}$ be as in the statement of the theorem. Given the regularity of $w$ and the fact that it solves \eqref{eq-fbp-SC}, we can apply It\^o-Meyer's formula for semimartingales to the Markov-modulated process $(e^{-\rho t} w(X^{\star}_{t}, \eps_{t}))_{t\geq0}$ on the time interval $[0,T]$ (see \cite{Eisenberg}) and obtain
\begin{align}
\label{ver-1}
&\eee_{(x,i)}\Big[e^{-\rho T} w(X^{I}_T,\eps_T)\Big] - w(x,i) \leq \eee_{(x,i)}\bigg[ - \int_0^T e^{-\rho t} \Big(\eta_{\varepsilon_t}X^I_t - c\big(X^{I}_t\big)^2 \Big) dt + \kappa \int_0^T e^{-\rho t} X^{I}_t dI^{c}_t\bigg] \nonumber \\
&\quad + \eee_{(x,i)}\bigg[\sum_{0\le t \le T} e^{-\rho t}\left(w(X^{I}_{t}, \eps_{t})-w(X^{I}_{t-},\eps_{t})\right)\bigg],
\end{align}
where $I^{c}$ denotes the continuous part of $I$. Then, by using that
\begin{align}
\label{sum-jumps-bis}
&\eee_{(x,i)}\bigg[\sum_{0\leq t \leq T} e^{-\rho t} \big(w(X^{I}_t, \eps_t)-w(X^{I}_{t-},\eps_t) \big)\bigg] = \eee_{(x,i)}\bigg[\sum_{0\leq t \leq T} e^{-\rho t} \mathds{1}_{\{\Delta I_{t}>0\}}\big(w(e^{\Delta I_{t}}X^{I}_{t-}, \eps_t)-w(X^{I}_{t-},\eps_{t})\big)\bigg] \nonumber\\
& = \eee_{(x,i)}\bigg[\sum_{0\leq t \leq T} e^{-\rho t} X^{I}_{t-} \int_0^{{\Delta I_{t}}} e^u w_x(e^u X^{I}_{t-}, \eps_{t}) du \bigg] \leq  \kappa 
\eee_{(x,i)}\bigg[\sum_{0\leq t \leq T} e^{-\rho t} X^{I}_{t-} \int_0^{{\Delta I_{t}}} e^u du \bigg] 
\end{align}
and recalling \eqref{eq-XcircI}, we obtain from \eqref{ver-1}, \eqref{sum-jumps-bis} and the growth condition of $w$ that
\begin{align}
\label{ver-2}
&w(x,i) \geq \eee_{(x,i)}\Big[e^{-\rho T} w(X^{I}_T,\eps_T)\Big] + \eee_{(x,i)}\bigg[\int_0^T e^{-\rho t} \Big(\eta_{\varepsilon_t}X^I_t - c\big(X^{I}_t\big)^2 \Big) dt - \kappa \int_0^T e^{-\rho t} X^{I}_t \circ dI_t\bigg] \nonumber \\
& \geq - \eee_{(x,i)}\Big[e^{-\rho T} K(1 + |X^{I}_T|^2)\Big] + \eee_{(x,i)}\bigg[\int_0^T e^{-\rho t} \Big(\eta_{\varepsilon_t}X^I_t - c\big(X^{I}_t\big)^2 \Big) dt - \kappa \int_0^T e^{-\rho t} X^{I}_t \circ dI_t\bigg]
\end{align}
Finally, by taking limits as $T\uparrow \infty$, and using that $I\in \mathcal{A}$, it yields
\begin{align}
\label{ver-3}
&w(x,i) \geq \eee_{(x,i)}\bigg[\int_0^{\infty} e^{-\rho t} \Big(\eta_{\varepsilon_t}X^I_t - c\big(X^{I}_t\big)^2 \Big) dt - \kappa \int_0^{\infty} e^{-\rho t} X^{I}_t \circ dI_t\bigg].
\end{align}
Since the latter inequality holds for any $I\in \mathcal{A}$ and any $(x,i) \in \rr_+ \times \{1,2\}$, we find $w \geq V$ on $\rr_+ \times \{1,2\}$.

On the other hand, picking $I = I^{\star} \in \mathcal{A}$, all the inequalities above become equalities (notice that $I^\star_{\cdot}$ is flat off $\{t\geq 0: X^{\star}_{t-} \leq a^{\star}_{\eps_t}\}$ and $X^{\star}_t \geq a_{\varepsilon_t}$ for all $t\geq 0$ a.s.) and we thus obtain
\begin{align}
\label{ver-4}
&w(x,i) = \eee_{(x,i)}\bigg[\int_0^T e^{-\rho t} \Big(\eta_{\varepsilon_t}X^{\star}_t - c\big(X^{\star}_t\big)^2 \Big) dt - \kappa \int_0^T e^{-\rho t} X^{\star}_t \circ dI^{\star}_t\bigg] \leq V(x,i), \qquad (x,i) \in \rr_+ \times \{1,2\}.
\end{align}
Hence, $w = V$ on $\rr_+ \times \{1,2\}$ and $I^{\star}$ as in \eqref{eq-optctrl-SC} is optimal. 
 \end{proof}

\begin{proof}[Proof of Proposition \ref{prop-verif-OS}]
We just provide details about the proof of the first and the third statements. The proof of (ii) can be indeed obtained by easily adapting Proposition 4.1 and Theorem 4.3 in \cite{FeRo} to the present setting, upon employing (i). In particular, the free-boundary problem \eqref{eq-fbp-OS} then follows by standard arguments based on strong Markov property.

As for (i), it is enough to observe that, for any $i\in\{1,2\}$, $x \mapsto v(x,i)$ is nonincreasing and that $\{(x,i)\in \rr_+ \times \{1,2\}:\, \eta_i - 2cx -(\rho + \delta) <0\} \subset \mathcal{C}$, since an integration by parts on \eqref{eq-pb-OS} yields
\begin{equation*}
v^Q(x,i) = \kappa + \inf_{\tau \geq 0} \bar{\eee}_{(x,i)}\bigg[ \int_0^\tau e^{-(\rho+\delta)s}\Big(\eta^Q_{\eps_s} - 2c\bar X_s -  (\rho+\delta) \Big) ds \bigg],
\end{equation*}
with $\bar X$ as in \eqref{eq-Xbar}. 

Given the continuity of $v(\cdot, i)$ obtained in (ii), optimality of $\tau^{\star}$ claimed in (iii) then follows by standard results in optimal stopping theory (see, e.g., \cite{PeskShir}).
\end{proof}


\begin{proof}[Proof of Proposition \ref{prop-conclus-step1}]
Set 
\begin{equation}
\label{eq-link-SC-OS}
w(x,i) := k_i + \int_{a_i}^x v(y,i) dy, \qquad\quad (x,i) \in \rr_+ \times \{1,2\},
\end{equation}
and notice that $k_1$ and $k_2$ are well defined, since the determinant of system \eqref{eq-k} is $\rho^2 + \rho(p_1+p_2) > 0$. Recalling that $(v(\cdot,i))_{i=1,2}$ is a solution to \eqref{eq-fbp-OS}, we want to prove that $(w(\cdot,i))_{i=1,2}$ is a solution to \eqref{eq-fbp-SC}. Then, employing Proposition \ref{prop-verif-SC} we conclude that $w=V$ on $\rr_+ \times \{1,2\}$.

Due to \eqref{eq-link-SC-OS} and Proposition \ref{prop-verif-OS}-(ii), for any $i=1,2$, $w(\cdot,i) \in C^2(\mathbb{R}_+)$ and it is such that there exists $K>0$ such that $|w(x,i)| \leq K(1 + |x|^2)$, for all $x \in \rr_+$. By \eqref{eq-fbp-OS} we clearly have $w'(x,i)=v(x,i) \leq \kappa $ for each $x \in \rr_+$ and, in particular, $w'(x,i)= \kappa$ for $x \leq a_i$. Also, notice that for $x > a_i$ we have
\begin{multline}
\label{eq-link1}
(\mathcal{L}-\rho)w(x,i) = \frac12 \sigma_i^2 x^2 v'(x,i) - \delta x v(x,i) - \rho k_i \\
+ \sum_{j=1}^2 P_{ij} k_j + \sum_{j=1}^2 P_{ij} \int_{a_j}^x \! v(y,j)dy - \rho \int_{a_i}^x \! v(y,i)dy.
\end{multline}
Now, from \eqref{eq-fbp-OS} it follows that, for each $y > a_i$,
\begin{equation}
\label{eq-link2}
-\rho v(y,i) = -\frac12 \sigma_i^2 y^2 v''(y,i) + (\delta-\sigma_i^2) y v'(y,i) + \delta v(y,i) - \sum_{j=1}^2 P_{ij} v(y,j) - \eta_i + 2cy.
\end{equation}
By an  integration by parts of \eqref{eq-link2} over the interval $(a_i,x)$, we get (recall that $v'(a_i,i)=0$, since $v'$ is continuous and $v'(a_i-,i)=0$, and that $v(a_i,i)=\kappa$)
\begin{equation}
\label{eq-link3}
-\rho \int_{a_i}^x \! v(y,i) dy = -\frac12 \sigma_i^2 x^2v'(x,i) + \delta (x v(x,i) - a_i \kappa) - \sum_{j=1}^2 P_{ij} \int_{a_i}^x \! v(y,j) dy - \eta_i(x-a_i) + c(x^2-a_i^2).
\end{equation}
Plugging equation \eqref{eq-link3} into equation \eqref{eq-link1}, we have
\begin{equation*}
(\mathcal{L}-\rho)w(x,i) =  -   \delta a_i \kappa - \rho k_i + \sum_{j=1}^2 P_{ij} k_j - \sum_{j=1}^2 P_{ij} \int_{a_i}^{a_j} \! v(y,j) dy - \eta_i(x-a_i) + c(x^2-a_i^2) .
\end{equation*}
By the definition of $k_i$ in \eqref{eq-k}, this is equivalent to 
\begin{equation*}
(\mathcal{L}-\rho)w(x,i) = - \eta_ix + cx^2.
\end{equation*}
On the other hand, the same argument applies if $x < a_i$, now leading to the inequality:
\begin{equation*}
(\mathcal{L}-\rho)w(x,i) \leq - \eta_ix + cx^2.
\end{equation*}
Hence, we proved that $w$ satisfies \eqref{eq-fbp-SC}. By Proposition \ref{prop-verif-SC}, whose assumptions are now clearly verified, and $I^{\star}$ as in \eqref{eq-optctrl-SC}, we conclude that $V=w$ on $\rr_+ \times \{1,2\}$.
\end{proof}


\begin{proof}[Proof of Proposition \ref{prop-stationary}]
With regards to the construction of $\Pi$ performed in Section \ref{ssec-step2}, it remains only to show that $\Pi(\cdot,i)$, $i\in\{1,2\}$, as in \eqref{eq-stationary-CDF} is in fact nondecreasing.

Let us first recall from Section \ref{ssec-step2} that 
\begin{equation}
\label{eqapp-ordercoeff1}
\theta_1<\theta_2<\alpha_2<0<\alpha_1.
\end{equation}
Furthermore, recall that function $\phi_1$ from Section \ref{ssec-step2} is given by 
\begin{equation*}
\phi_1(\alpha) = \frac12 \sigma_1^2\alpha^2 + \Big(\delta + \frac12 \sigma^2_1\Big)\alpha - p_1.
\end{equation*}
In particular, notice that
\begin{equation}
\label{eqapp-ordercoeff2}
\phi_1(\theta_1)>0, \qquad\qquad \phi_1(\theta_2)<0,
\end{equation}
which follows by \eqref{eqapp-ordercoeff1} and the fact that $\phi_1$ is convex, negative in $(\alpha_2,\alpha_1)$ and positive elsewhere.
Finally, the solution to system \eqref{eq-stationary-system} is given by 
\begin{equation*}
A_1 = \frac{p_2 \widehat{A}_1}{(p_1+p_2) D}, \qquad 
A_2 = -A_1, \qquad 
B_1 = \frac{p_2 \widehat{B}_1}{(p_1+p_2) D}, \qquad 
B_2 = \frac{p_2 \widehat{B}_2}{(p_1+p_2) D},  
\end{equation*}
where the numerators are given by (we have used the explicit expressions of $\pi_1$ and $\pi_2$) 

%
%
%

\begin{equation*}
\begin{aligned}
& \widehat{A}_1 := e^{(\alpha_1 + \alpha_2) b_1} \Big[(\theta_1-\theta_2) p_1 - \theta_2  \phi_1(\theta_1) + \theta_1 \phi_1(\theta_2) \Big], \\
& \widehat{B}_1 := e^{\alpha_1 b_2 + \alpha_2 b_1} \Big[(\alpha_1 - \theta_2) p_1 + \alpha_1 \phi_1(\theta_2) \Big] - e^{\alpha_1 b_1 + \alpha_2 b_2} \Big[ (\alpha_2 - \theta_2) p_1 + \alpha_2 \phi_1(\theta_2) \Big], \\
& \widehat{B}_2 := -e^{\alpha_1 b_2 + \alpha_2 b_1} \Big[ (\alpha_1 - \theta_1) p_1 + \alpha_1 \phi_1(\theta_1) \Big] + e^{\alpha_1 b_1 + \alpha_2 b_2} \Big[ (\alpha_2 - \theta_1) p_1 + \alpha_2 \phi_1(\theta_1) \Big],
\end{aligned}    
\end{equation*}
and the common denominator by
\begin{multline*}
D := e^{\alpha_1 b_2 + \alpha_2 b_1} \Big[ \theta_1 \phi_1(\theta_2) - \theta_2 \phi_1(\theta_1) + \alpha_1 (\phi_1(\theta_1) - \phi_1(\theta_2)) \Big] \\
+ e^{\alpha_1 b_1 + \alpha_2 b_2} \Big[ \theta_2 \phi_1(\theta_1) - \theta_1 \phi_1(\theta_2) + \alpha_2 (-\phi_1(\theta_1) + \phi_1(\theta_2)) \Big].
\end{multline*}

By the definition of $\Pi$ in \eqref{eq-stationary-system}, to prove that its components are non-decreasing it suffices to check $\Pi(\cdot, 1)$ in $(b_1,b_2)$, $\Pi(\cdot, 1)$ in $(b_2,\infty)$ and $\Pi(\cdot, 2)$ in $(b_2,\infty)$. Let us first consider $\Pi(\cdot, 1)$ in $(b_1,b_2)$. By \eqref{eq-stationary-system}(i) and \eqref{eq-stationary-CDF}, in $(b_1,b_2)$ we have
\begin{equation*}
\Pi'(z,1) = \alpha_1 A_1 e^{\alpha_1 (z - b_1)} + \alpha_2 A_2 e^{\alpha_2 (z - b_1)} = A_1 \big( \alpha_1 e^{\alpha_1(z - b_1)} - \alpha_2 e^{\alpha_2(z-b_1)} \big),  
\end{equation*}
so that $\Pi'(z,1) \geq 0$ if and only if $A_1>0$ (recall that $\alpha_2 <0<\alpha_1$). As for the denominator of $A_1$, one can rewrite
\begin{multline*}
D = e^{\alpha_1 b_2 + \alpha_2 b_2} \big( \theta_1 \phi_1(\theta_2) - \theta_2 \phi_1(\theta_1) \big) \big( e^{-\alpha_2 (b_2 - b_1)} - e^{-\alpha_1(b_2 - b_1)} \big) \\
+ \big( \phi_1(\theta_1) - \phi_1(\theta_2) \big) \big( \alpha_1 e^{\alpha_1 b_2 + \alpha_2 b_1} - \alpha_2 e^{\alpha_1 b_1 + \alpha_2 b_2)} \big)>0,       
\end{multline*}
where the positivity follows by \eqref{eqapp-ordercoeff1} and \eqref{eqapp-ordercoeff2}. As for the numerator of $A_1$, by the formulas for $\pi_i$ in \eqref{eq-Pi}, we have
\begin{equation*}
\widehat{A}_1>0 \Longleftrightarrow f(\theta_1)<0, 
\qquad\quad \text{with} \qquad\quad 
f(\theta) := \theta_2 \phi_1(\theta) - \theta \phi_1(\theta_2) + p_1(\theta_2 -\theta), 
\end{equation*}
which is verified: Indeed, we have $f(\theta_2)=0$, $f(-\infty)=-\infty$, $f'(\theta_2)>0$, $f''<0$ in $(-\infty,0)$, from which $f \leq 0$ in $(-\infty, \theta_2) \ni \theta_1$.

Let us now consider $\Pi(\cdot, 1)$ in $(b_2,\infty)$. By plugging \eqref{eq-stationary-system}(i) into \eqref{eq-stationary-system}(iii), we get
\begin{equation*}
\theta_1 B_1 + \theta_2 B_2 = A_1 (\alpha_1e^{\alpha_1 (b_2 - b_1)} - \alpha_2e^{\alpha_2 (b_2 - b_1)}) > 0,   
\end{equation*}
where the positivity of the right-hand side follows from \eqref{eqapp-ordercoeff1}. In particular, in $(b_2,\infty)$ we have 
\begin{equation*}
\Pi'(z,1) = \theta_1 B_1 e^{\theta_1 (z - b_2)} + \theta_2 B_2 e^{\theta_2 (z - b_2)} > \theta_2 B_2 \big(e^{\theta_2(z - b_2)} - e^{\theta_1(z - b_2)} \big),    
\end{equation*}
so that $\Pi'(z,1) \geq 0$ if and only if $B_2<0$. We have already proved that the denominator $D$ is positive, so it remains to prove that $\widehat{B}_2<0$. This condition is indeed verified, since $\widehat{B}_2$ can be rewritten as
\begin{multline*}
\hat B_2 = p_1 e^{\alpha_1 b_1 + \alpha_2 b_1}\big[(\alpha_2 - \theta_1) e^{\alpha_2(b_2 - b_1)} - (\alpha_1 - \theta_1) e^{\alpha_1(b_2 - b_1)} \big] \\
+ \phi_1(\theta_1) \big[ \alpha_2 e^{\alpha_1 b_1 + \alpha_2 b_2} - \alpha_1 e^{\alpha_2 b_1 + \alpha_2 b_1} \big]
\end{multline*}
and both the terms on the right-hand side are negative: the second one by \eqref{eqapp-ordercoeff1}, the first one since
\begin{equation*}
(\alpha_2 - \theta_1) e^{\alpha_2(b_2 - b_1)} - (\alpha_1 - \theta_1) e^{\alpha_1(b_2 - b_1)} <  (\alpha_1-\theta_1) \big(e^{\alpha_2(b_2 - b_1)} - e^{\alpha_1(b_2 - b_1)} \big)<0.   
\end{equation*}

Let us finally consider $\Pi(\cdot, 2)$ in $(b_2,\infty)$. We have
\begin{equation*}
\Pi'(z,2) = - \theta_1 B_1 \frac{\phi_1(\theta_1)}{p_2} e^{\theta_1 (z - b_2)} - \theta_2 B_2 \frac{\phi_1(\theta_2)}{p_2} e^{\theta_2 (z - b_2)}. 
\end{equation*}
so that $\Pi'(z,2) \geq 0$ if and only if 
\begin{equation*}
B_1 \geq - \Big(B_2 \frac{\phi_1(\theta_2)}{\phi_1(\theta_1)} \frac{\theta_2}{\theta_1} \Big) e^{(\theta_2 - \theta_1) (z-b_2)} .
\end{equation*}
Since the parenthesis in the right-hand side is positive and since $z-b_2 >0$, it suffices to prove
\begin{equation*}
B_1 \geq - B_2 \frac{\phi_1(\theta_2)}{\phi_1(\theta_1)} \frac{\theta_2}{\theta_1},
\end{equation*}
which is equivalent to
\begin{equation*}
e^{(\alpha_1 - \alpha_2) (b_2 - b_1)} F(\alpha_1) \leq F(\alpha_2),
\end{equation*}
where we have set
\begin{equation*}
F(\alpha) := \alpha \Big[ \phi_1(\theta_1) \phi_1(\theta_2) (\theta_1 - \theta_2) + p_1 \big(\theta_1 \phi_1(\theta_1) - \theta_2 \phi_1(\theta_2)\big)\Big] + p_1 \theta_1 \theta_2 (\phi_1(\theta_2) - \theta_2 \phi_1(\theta_1)).
\end{equation*}
Notice now that it suffices to prove that $F'(\alpha)<0$ for each $\alpha \in \rr$. Indeed, if that is the case, it follows that $F(\alpha_1) \leq F(\alpha_2)$ and that $F(\alpha_1) \leq F(0) < 0$, so that 
\begin{equation*}
e^{(\alpha_1 - \alpha_2) (b_2 - b_1)} F(\alpha_1) \leq F(\alpha_1) \leq F(\alpha_2).
\end{equation*}
Let us then verify that $F$ is decreasing, which is equivalent to 
\begin{equation}
\label{eqapp-the-last}
\theta_1 \phi_1(\theta_1) \big(\phi_1(\theta_2) + p_1\big) < \theta_2 \phi_1(\theta_2) \big(\phi_1(\theta_1) + p_1\big).
\end{equation}
However, \eqref{eqapp-the-last} is always verified. Indeed, if $\phi_1(\theta_2) + p_1 >0$, then \eqref{eqapp-the-last} is trivially true. On the other hand, if $\phi_1(\theta_2) + p_1 < 0$, then \eqref{eqapp-the-last} is equivalent to
\begin{equation*}
\frac{\phi_1(\theta_1)}{\phi_1(\theta_2)} > \frac{\theta_1}{\theta_2} \, \frac{\phi_1(\theta_1) + p_1}{\phi_1(\theta_2)+p_1}.
\end{equation*}
Considering the right-hand side, the first factor is greater than 1, while the second factor is negative. Hence,
\begin{equation*}
\frac{\theta_1}{\theta_2} \, \frac{\phi_1(\theta_1) + p_1}{\phi_1(\theta_2)+p_1} < \frac{\phi_1(\theta_1) + p_1}{\phi_1(\theta_2)+p_1} < \frac{\phi_1(\theta_1)}{\phi_1(\theta_2)},
\end{equation*}
where the last inequality immediately follows from \eqref{eqapp-ordercoeff2}.
\end{proof}

\bibliographystyle{plain}

\end{document}